\documentclass[reqno,twoside,12pt]{amsart}

\hoffset - 1.5cm

\usepackage{amsmath}
\usepackage{amssymb}
\usepackage{graphicx}
\usepackage{amstext}
\usepackage{amsopn}
\usepackage{amsfonts}
\usepackage{xcolor}
\usepackage{setspace}
\usepackage{enumerate}
\usepackage[polish,english]{babel}
\usepackage[utf8]{inputenc}
\usepackage{polski}

\usepackage{lmodern}

\newtheorem{Remark}{Remark}[section]
\newtheorem{Corollary}[Remark]{Corollary}
\newtheorem{Definition}[Remark]{Definition}
\newtheorem{Example}[Remark]{Example}
\newtheorem{Fact}[Remark]{Fact}
\newtheorem{Lemma}[Remark]{Lemma}
\newtheorem{Proposition}[Remark]{Proposition}
\newtheorem{Theorem}[Remark]{Theorem}

\newcommand{\ba}{\begin{array}}
\newcommand{\bc}{\begin{center}}
\newcommand{\bd}{\begin{description}}
\newcommand{\bdm}{\begin{displaymath}}
\newcommand{\be}{\begin{enumerate}}
\newcommand{\beq}{\begin{equation}}
\newcommand{\bdf}{\begin{Definition}}
\newcommand{\bex}{\begin{Example}}
\newcommand{\bft}{\begin{Fact}}
\newcommand{\bl}{\begin{Lemma}}
\newcommand{\bp}{\begin{Proposition}}
\newcommand{\br}{\begin{Remark}}
\newcommand{\bt}{\begin{Theorem}}
\newcommand{\bco}{\begin{Corollary}}
\newcommand{\bhy}{\begin{Hypothesis}}
\newcommand{\ea}{\end{array}}
\newcommand{\ec}{\end{center}}
\newcommand{\ed}{\end{description}}
\newcommand{\edm}{\end{displaymath}}
\newcommand{\ee}{\end{enumerate}}
\newcommand{\eeq}{\end{equation}}
\newcommand{\edf}{\end{Definition}}
\newcommand{\eex}{\end{Example}}
\newcommand{\eft}{\end{Fact}}
\newcommand{\el}{\end{Lemma}}
\newcommand{\ep}{\end{Proposition}}
\newcommand{\er}{\end{Remark}}
\newcommand{\et}{\end{Theorem}}
\newcommand{\eco}{\end{Corollary}}
\newcommand{\ehy}{\end{Hypothesis}}

\newcommand{\bH}{\mathbb{H}}
\newcommand{\bI}{\mathbb{I}}

\newcommand{\bN}{\mathbb{N}}

\newcommand{\bR}{\mathbb{R}}

\newcommand{\bV}{\mathbb{V}}
\newcommand{\bW}{\mathbb{W}}
\newcommand{\bX}{\mathbb{X}}
\newcommand{\bY}{\mathbb{Y}}
\newcommand{\bZ}{\mathbb{Z}}

\newcommand{\cC}{\mathcal{C}}

\newcommand{\cL}{\mathcal{L}}

\newcommand{\cN}{\mathcal{N}}

\newcommand{\cR}{\mathcal{R}}

\newcommand{\cT}{\mathcal{T}}

\newcommand{\cV}{\mathcal{V}}

\newcommand{\im}{\mathrm{ im \;}}

\numberwithin{equation}{section} \errorcontextlines=0

\newcommand{\cl}{\mathrm{cl}}

\newcommand{\sone}{S^1}

\newcommand{\ds}{\displaystyle}

\newcommand{\sub}{\overline{\mathrm{sub}}}

\newcommand{\chittwo}{\chi_{T^2}}

\newcommand{\db}{\mathrm{deg}_{\mathrm{B}}}

\newcommand{\dg}{\mathrm{deg}^{\nabla}_G}

\newcommand{\degsone}{\mathrm{deg}^{\nabla}_{S^1}}
\newcommand{\degttwo}{\mathrm{deg}^{\nabla}_{T^2}}
\newcommand{\ib}{\mathrm{i}_{\mathrm{B}}}

\newcommand{\bif}{\mathcal{BIF}}
\newcommand{\bifttwo}{\mathcal{BIF}_{T^2}}

\newcommand{\h}{\mathbb{H}}

\newcommand{\hone}{\mathbb{H}^1_{2\pi}}

\makeatletter
\@namedef{subjclassname@2020}{%
  \textup{2020} Mathematics Subject Classification}
\makeatother

\usepackage{todonotes}
\usepackage{setspace}

\begin{document}

\title[Autonomous $S^1$-symmetric Newtonian systems]{Periodic solutions of autonomous $S^1$-symmetric Newtonian systems}

\author{Anna Go{\l}\c{e}biewska}
\address{Faculty of Mathematics and Computer Science \\ Nicolaus Copernicus University in Toru\'n\\
PL-87-100 Toru\'{n} \\ ul. Chopina $12 \slash 18$ \\
Poland}
\author{S{\l}awomir Rybicki}

\author{Piotr Stefaniak}
\email{aniar@mat.umk.pl (A. Go{\l}\c{e}biewska)}
\email{rybicki@mat.umk.pl (S. Rybicki)}
\email{cstefan@mat.umk.pl (P. Stefaniak)}


\keywords{symmetric Newtonian systems, bifurcations of periodic solutions, equivariant bifurcation theory}
\subjclass[2020]{Primary: 37J46, 37G40; Secondary: 37J20, 70K42}

\begin{abstract}
The aim of this paper is to formulate necessary conditions and sufficient ones for the existence of  closed connected sets of nonstationary $2 \pi$-periodic solutions of $\sone$-symmetric Newtonian systems in  $C_{2 \pi}([0,2\pi],\Omega) \times (0,+ \infty)$. As the main topological tool we apply the degree for equivariant gradient maps.
\end{abstract}
\maketitle

\section{Introduction}
 Let us consider a Newtonian system of the form
\beq\label{nsin}
\ddot  u(t)=-\lambda^2 U'(u(t)),\eeq where $\lambda > 0,\Omega \subset \bR^N$ is open, $ U \in C^2(\Omega,\bR)$,  and $ (U')^{-1}(0)$ is finite.

We are interested in finding continua  of nonstationary $2\pi$-pe\-rio\-dic solutions of the system \eqref{nsin} emanating from the set of stationary ones.

The method that can be applied to this problem is to consider some Hamiltonian system corresponding to the system \eqref{nsin}. More precisely, the Newtonian system \eqref{nsin} is equivalent to the following Hamiltonian one
\beq\label{hsin}
\dot  x(t)=\lambda JH'(x(t)),
\eeq
where $\lambda > 0, x=(u,v)$, the Hamiltonian $H \in C^2(\bR^{2N},\bR)$ is defined by $H(u,v) = -\frac{1}{2} |v|^2 - U(u)$ and $J$ is the standard symplectic matrix.
Moreover, it is clear that $u_{0} \in (U')^{-1}(0)$ iff $(u_{0},0) \in (H')^{-1}(0)$.

The continua of nonstationary $2 \pi$-periodic solutions of autonomous Hamiltonian systems of the form \eqref{hsin} (with an arbitrary Hamiltonian $H$ having a finite number of critical points) have been studied by Dancer and the second author, see Theorem 3.3 of \cite{danryb}.
In particular, for any $(u_{0},0) \in (H')^{-1}(0)$ and $\lambda_0 > 0$ there  has been defined a bifurcation index $\ds \eta((u_{0},0),\lambda_0)$, which is an element of the Euler ring $U(\sone)$ of the group $\sone$ (we refer the reader to \cite{TomDieck1979}, \cite{TomDieck} for the definition and properties of the Euler ring of a compact Lie group $G$), see the formulas (3.4) and  (3.5) in \cite{danryb}. The nontriviality  of this index implies  a bifurcation of continua of nonstationary $2 \pi$-periodic  solutions of the system \eqref{hsin} from  $((u_{0},0),\lambda_0)$. Moreover,  bounded continua satisfy the summation formula, see the formula (3.6) in \cite{danryb}.

We emphasize that if the bifurcation index $\ds \eta((u_{0},0),\lambda_0)  \in U(\sone)$ is nontrivial then the Brouwer index of the isolated critical point $(u_{0},0)$ of Hamiltonian $H$, denoted by  $\ib((u_{0},0),H') \in \bZ$,   is nonzero. It is clear that  $\ib((u_{0},0),H')=(-1)^N \ib(u_{0},-U')$. Therefore,  nontriviality of $\ds \eta((u_{0},0),\lambda_0)$ implies $\ib(u_{0},-U') \neq 0$.
Hence, it would be desirable to know sufficient conditions for the existence of continua of nonstationary $2 \pi$-periodic solutions of the system \eqref{nsin} emanating from the stationary solution $(u_{0},\lambda_0)$ satisfying $\ib(u_{0},-U')=0$. These results, in the general setting, are  far from being conclusive. However, if we restrict the considered problem to some special situation, we can use more sophisticated tools. In particular additional symmetries of the system \eqref{nsin} allow us to define a more efficient bifurcation index.

That is why  in this article we consider $\sone$-symmetric Newtonian systems of the form \eqref{nsin}. To be more precise, we consider $\bR^N$ as an orthogonal representation of the group $\sone$, assume that $\Omega \subset \bR^N$ is open and $\sone$-invariant and that the potential $U \in C^2(\Omega,\bR)$ is $\sone$-invariant, see Section \ref{sec:formulation}. Since the potential $U$ is $\sone$-invariant, the gradient $U' \in C^1(\Omega,\bR^N)$ is $\sone$-equivariant. Therefore one can define an index $\mathrm{i}_{\nabla}^{\sone}(u_{0},-U') \in U(\sone)$ by   $\mathrm{i}_{\nabla}^{\sone}(u_{0},-U') = \deg^{\nabla}_{\sone}(-U',\Omega_{0})$, where $ \deg^{\nabla}_{\sone}(\cdot,\cdot)$ is the degree for $\sone$-equivariant gradient maps (see Section \ref{degg} as well as \cite{Geba} and \cite{Ryb2005milano} for the definition and properties of this degree) and $ \Omega_{0} \subset \Omega$ is an open $\sone$-invariant subset such that
$\cl (\Omega_{0}) \cap (U')^{-1}(0) = \{u_{0}\}$.
It is known that if $\ib(u_{0},-U')$  is nonzero then $\mathrm{i}_{\nabla}^{\sone}(u_{0},-U')$ is nontrivial. However, the opposite implication is not true.

The main result of our paper concerns the problem of
 finding necessary conditions and sufficient ones for the existence of continua of nonstationary $2 \pi$-periodic  solutions of the $\sone$-symmetric Newtonian system \eqref{nsin} in $C_{2 \pi}([0,2\pi],\Omega) \times (0,+\infty)$.
This result is formulated in Theorem \ref{thm:main}. The basic idea of the proof of this theorem is to apply the degree for equivariant gradient maps.
It is worth pointing out that our result is also applicable in a case when the results of \cite{danryb} cannot be used i.e. when $\ib(u_{0},-U') = 0$,   see  Lemma \ref{lem:exam} and Remark \ref{rembg}.

Let us emphasize once again  that in our paper we have replaced the Brouwer index $\ib(u_{0},-U') \in \bZ$ used in \cite{danryb} by the index $\mathrm{i}_{\nabla}^{\sone}(u_{0},-U') \in U(\sone)$ defined for the class of $\sone$-equivariant gradient maps.
The choice of the degree for equivariant gradient maps seems to be the best adapted to our theory.

It would be interesting to consider also $\Gamma$-symmetric Newtonian and Hamiltonian systems, where $\Gamma$ is an arbitrary compact Lie group. However, this topic exceeds the scope of this paper.

After this introduction our paper is organized as follows.

In Section \ref{preres} we introduce a notion of an orthogonal representation of a compact Lie group $G$, see Subsection \ref{equivthe}, and discuss some properties of orthogonal representations of the group $\sone$ and the torus $T^2$, see Subsections \ref{repsone}, \ref{reptor}, respectively. Some properties of the Euler ring $U(T^2)$ are discussed in Subsection \ref{sec:Euler}. We finish this section with Subsection \ref{degg} containing some properties of the degree for $G$-equivariant gradient maps, with special emphasis on groups $\sone, T^2$.

 In Section \ref{sec:main} we formulate and prove the main result of this article. More precisely speaking, in Subsection \ref{sec:formulation} we formulate the main result of our paper, see Theorem \ref{thm:main}. In Subsection \ref{sec:var} we define a $T^2$-symmetric $C^2$-functional $\Phi \colon \bH^1_{2\pi} \times (0,+\infty) \to \bR$ whose critical points are in one-to-one correspondence with $2 \pi$-periodic solutions of the $\sone$-symmetric Newtonian system \eqref{nsin}. Subsection \ref{sec:glob} contains a discussion of global bifurcation of $T^2$-orbits of solutions of the equation $\nabla_u \Phi(u,\lambda) = 0$. In Subsection \ref{sec:proofmain} we present a proof of Theorem \ref{thm:main}.

The discussion of the results is presented in Section \ref{frc}. In Subsection \ref{ceqb} we define a particular $\sone$-invariant potential by the formula \eqref{eq:exampU}. We show that this potential satisfies the assumptions of our main theorem, whereas it does not fulfill the assumptions of Theorem 3.3 of  \cite{danryb}. In Subsection \ref{ncco} we look more closely at continua of nonstationary $2\pi$-periodic solutions of the $\sone$-symmetric Newtonian system \eqref{nsin} with the potential $U$ defined by the formula \eqref{eq:exampU}. We prove that all the continua are not compact in $C_{2 \pi}([0,2\pi],\Omega)  \times (0,+\infty)$.
We finish the article with a theorem concerning the existence of non-compact continua, see Theorem \ref{thm:unbounded}.

\section{Preliminary results}
\label{preres}

In this section we recall basic notion that we will use in this article: we briefly outline basics of representation theory, notion of the Euler ring and of the degree for equivariant gradient maps. Additionally, we prove some auxiliary properties of the Euler ring of torus and a characterization of some spaces as representations of torus.

\subsection{Notation and elements of equivariant theory}
\label{equivthe}

We begin by recalling some terminology and facts concerning group actions and group representations. Let $G$ be a compact Lie group and let $e$ be its neutral element.

We say that $G$ acts on a topological space $\bX$  (or $\bX$ is a $G$-space) if there is a continuous map $\mu\colon G \times \bX \to \bX$ such that
\be
\item $\mu(e,x)=x$,
\item  $\mu(g_1,\mu(g_2,x)) = \mu (g_1 g_2,x)$,
\ee
for all $g_1, g_2 \in G, x \in \bX$. In what follows, for simplicity of notation, we write $gx$ instead of $\mu(g,x)$.

For a given $x \in \bX$ the subgroup $G_x=\{g \in G \colon gx=x\}$ is called the isotropy group (or the stabilizer) of $x$ and the set $G(x) = \{gx \colon g \in G\}$ is called the orbit of $x$. A subset $\Omega \subset \bX$ is said to be $G$-invariant, if $gx \in \Omega$ for all $g \in G, x \in \Omega$ i.e. $G(x) \subset \Omega$ for any $x \in \Omega$.

Let $\bX, \bY$ be two $G$-spaces. We say that a continuous map $f\colon \bX \to \bY$ is $G$-equivariant (or $f$ is a $G$-map) if $f(gx)=gf(x)$ for all $g \in G, x \in \bX$. If $\bY$ is a trivial $G$-space, i.e. $gy=y$ for all $g \in G, y \in \bY$, then a $G$-equivariant map $f\colon \bX  \to \bY$ is said to be $G$-invariant.

An orthogonal representation of $G$ (or a $G$-representation) is a pair $\bV = (\bV_0, \rho_{\bV})$, where $\bV_0$ is a real, linear space and $\rho_{\bV}\colon  G \to O(\bV_0)$ is a continuous homomorphism from $G$ into the group of all linear, orthogonal automorphisms of $\bV_0$. If $\bV_0=\bR^n$, we use the standard notation $O(n,\bR)$ for $O(\bR^n)$.  We often do not distinguish  between $\bV$ and  $\bV_0$ using the same letter for the representation and the underlying linear space.

\begin{Remark}\label{rem:actrepr}
Note that if $\bV $ is a $G$-representation then letting $g v = \rho_{\bV}(g) v$ we obtain a $G$-action  on $\bV$. On the other hand, if $\bV$ is a linear space with an inner product, the $G$-representation can be understood as an action of $G$ on $\bV$ such that for each $g \in G$ the map $v \mapsto gv$ is a linear, orthogonal map, since such an action defines a homomorphism  $\rho_{\bV}\colon  G \to O(\bV)$ given by $\rho_{\bV}(g)(v)=gv$ (see \cite{BroDie}).
\end{Remark}

Two $G$-representations $\bV, \bW$ are called equivalent if there is a $G$-equivariant linear isomorphism $T\colon \bV \to \bW$. Given two $G$-representations  we denote by $\bV \oplus \bW$ the direct sum of $\bV$ and $\bW$ i.e. the direct sum of linear spaces with the linear group action defined by $g(v,w)=(gv, gw)$.

A $G$-representation $\bV$ is called irreducible if it  has no nontrivial invariant subspaces.

Let $B_\delta(v_0,\bV)=\{v \in \bV_0 \colon \|v-v_0\| < \delta\}.$ For simplicity of notation we write  $B_\delta(\bV)$ instead of $B_\delta(0,\bV)$.

\subsection{$S^1$-Representations}
\label{repsone}

Consider the group $S^1=\{e^{i\phi}\colon \phi\in\bR\}$ and fix $m \in \bN$. Define a homomorphism $\rho_m \colon S^1 \to O(2,\bR)$ by
\begin{equation}\label{eq:rhom}
\rho_m(e^{i\phi})=\left[
\begin{array}{rr} \cos(m \phi) &-\sin(m \phi)\\
\sin(m \phi) &\cos(m \phi)
\end{array}
\right]
\end{equation}
and $\rho_0\colon S^1\to O(1,\bR)=\{\pm1\}$ by $\rho_0(e^{i \phi}) = 1.$

Denote by $\bR[1,m]$ the real two-dimensional $\sone$-representation $(\bR^2,\rho_m)$ and by $\bR[1,0]$ the real one-dimensional $\sone$-representation $(\bR,\rho_0).$ 
Additionally, denote by $\bR[k,m]$ (by $\bR[k,0]$, respectively) the direct sum of $k$ copies of $\bR[1,m]$ (of $\bR[1,0]$, respectively). The following fact is well-known, see \cite{Adams}.

\begin{Fact}
Any real irreducible $S^1$-representation  is equivalent to $\bR[1,m]$, for $m \in \bN$, or $\bR[1,0]$. Moreover, representations $\bR[1,m]$ and $\bR[1,m']$ are equivalent iff $m=m'$.
\end{Fact}

As the consequence of the above fact we obtain the well-known description of finite-dimensional orthogonal representations of $S^1$.
\begin{Corollary}\label{cor:s1repr}
Let $\bV=(\bR^N, \rho_{\bV})$ be an orthogonal representation of  $S^1$. Then there exist $r \in \bN\cup \{0\}$ and finite sequences $(k_0, \ldots, k_r),(m_1, \ldots, m_r),$ where $k_1, \ldots, k_r, m_1, \ldots, m_r \in \bN, k_0 \in \bN \cup \{0\}, m_1<\ldots<m_r$ such that $\bV$ is equivalent to  $\bR[k_0,0] \oplus \bR[k_1,m_1]\oplus \ldots \oplus \bR[k_r,m_r].$
\end{Corollary}

\subsection{$T^2$-Representations}
\label{reptor}

Denote by $T^2=\{(e^{i\phi_1}, e^{i\phi_2})\in S^1\times S^1\colon \phi_1,\phi_2\in\bR\}$ the 2-dimensional torus.
Fix $(m,n) \in \bZ^2\setminus \{(0,0)\}$ and define a homomorphism $\rho_{(m,n)} \colon T^2 \to O(2,\bR)$ by
\begin{equation*}
\rho_{(m,n)}(e^{i\phi_1}, e^{i\phi_2})=\left[
\begin{array}{rr} \cos(m \phi_1+n\phi_2) &-\sin(m \phi_1+n\phi_2)\\
\sin(m \phi_1+n\phi_2) &\cos(m \phi_1+n\phi_2)
\end{array}
\right].
\end{equation*}
Moreover, put $\rho_{(0,0)}(e^{i\phi_1}, e^{i\phi_2})=1.$ Denote by $\bR[1,(m,n)]$ the real two-dimensional representation $(\bR^2, \rho_{(m,n)})$ and by $\bR[1,(0,0)]$ the real one-dimensional representation $(\bR, \rho_{(0,0)}).$ Additionally, denote by $\bR[k,(m,n)]$ (respectively by $\bR[k,(0,0)])$ the direct sum of $k$ copies of $\bR[1,(m,n)]$ (respectively $\bR[1,(0,0)]$).

The following fact can be found in \cite[Proposition II.8.5]{BroDie}.

\begin{Fact}
Any real irreducible $T^2$-representation  is equivalent to $\bR[1,(m,n)]$ (for $(m,n) \in \bZ^2\setminus \{(0,0)\}$) or $\bR[1,(0,0)]$. Moreover, representations $\bR[1,(m,n)]$ and $\bR[1,(m',n')]$ are equivalent iff $(m,n)=\pm(m',n')$.
\end{Fact}

Using this fact we recall the well-known description of finite-dimensional orthogonal $T^2$-representations, similar to the one for $S^1$-representations given in Corollary \ref{cor:s1repr}.
\begin{Corollary}\label{cor:repT}
Let $\bV=(\bR^N, \rho_{\bV})$ be an orthogonal representation of  $T^2$. Then there exist $r \in \bN\cup \{0\}$, $k_0\in\bN\cup\{0\}$, $k_1, \ldots, k_r\in \bN$ and $(m_1,n_1),\ldots, (m_r,n_r)\in\bZ^2\setminus\{(0,0)\}$ such that $\bV$ is equivalent to $\bR[k_0,(0,0)]\oplus \bR[k_1,(m_1,n_1)]\oplus\ldots \oplus \bR[k_r,(m_r,n_r)].$
\end{Corollary}

In the proofs of our main results we are going to study some function spaces and their structures as $T^2$-representations. In particular, we are interested in the spaces $\h_n$ given by
\begin{equation}\label{eq:defHn}
\h_n =\{u \colon [0, 2 \pi] \to \bR^N; u(t)=a\cos nt+b\sin nt\colon a,b\in\bR^N \}
\end{equation}
for $n\geq 0$, with the inner product defined by $\ds \langle u,v\rangle_{\h_n} = \int_0^{2\pi} (\dot u(t), \dot v(t)) + (u(t),v(t)) \; dt.$  
Let $\bR^N$ be an $S^1$-representation equivalent to
\begin{equation*}
\bV= \bR[k_0,0] \oplus \bR[k_1, m_1] \oplus \ldots \oplus \bR[k_r,m_r].
\end{equation*}
Then, we define a $T^2$-action on $\bH_n$ by
\begin{equation}\label{eq:actionHn}
(e^{i\phi_1}, e^{i\phi_2})(u)(t)=\rho_{\bV}(e^{i \phi_1})(u(t+\phi_2)).
\end{equation}
Observe that the space $\bH_n$ with this action is an orthogonal representation of $T^2$, see Remark \ref{rem:actrepr}.

\begin{Remark}\label{rem:H0}
Obviously, in the case $n=0$, the space $\bH_0$ consists of constant functions and can be identified with $\bR^N$. Therefore one can see the space $\bH_0$ in two ways: as an $S^1$-representation $\bV$  and as a $T^2$-representation  with the action given by
\begin{equation*}
(e^{i\phi_1}, e^{i\phi_2})u_0=\rho_{\bV}(e^{i \phi_1})u_0.
\end{equation*}
In particular the isotropy group $T^2_{u_0}$ of $u_0 \in \bH_0$ is $S^1_{u_0} \times S^1$, where $S^1_{u_0}$ is the isotropy group of $u_0 \in \bR^{N}.$
\end{Remark}

\begin{Lemma}\label{Lem:Hn}
Consider the space $\bH_n$ with the $T^2$-action given by the formula \eqref{eq:actionHn}. Then, for $n>0$
\begin{equation}\label{eq:reprHn}
\bH_n \equiv_{T^2} \bR[k_0,(0,n)]\oplus \bigoplus_{j=1}^r\big(\bR[k_j,(m_j,n)] \oplus\bR[k_j, (-m_j,n)]\big).
\end{equation}
Moreover,
\begin{equation}\label{eq:reprH0}
\bH_0 \equiv_{T^2} \bR[k_0,(0,0)] \oplus \bR[k_1,(m_1,0)] \oplus \ldots \oplus \bR[k_r,(m_r,0)].
\end{equation}
\end{Lemma}

\begin{proof}
Fix $m\geq 0$ and consider the space $E_n=\{a \cos nt + b \sin nt\colon a,b \in \bR[1,m]\}$. Define the action of the group $T^2$ on the elements of $E_n$ by
$$(e^{i\phi_1},e^{i\phi_2})(a \cos nt + b \sin nt)=(\rho_m(e^{i\phi_1}) a) \cos n(t+\phi_2)+(\rho_m(e^{i\phi_1}) b) \sin n(t+\phi_2),$$
where $\rho_m$ is given by the formula \eqref{eq:rhom}. Observe that this action coincides with the formula \eqref{eq:actionHn} for $\bV=\bR[1,m]$.  Note that in the case $m>0$ the space $E_n$ can be decomposed as a  direct sum of two subspaces $E_n^1,E_n^2 \subset E_n$ given by
\begin{equation*}
\begin{split}
&E_n^1=\left\{\left(\begin{array}{c}a_1\\
a_2\end{array}\right) \cos nt + \left(\begin{array}{r}-a_2\\ a_1\end{array}\right) \sin nt \colon a_1,a_2\in\bR \right\},\\
&E_n^2=\left\{\left(\begin{array}{r}a_1\\ -a_2\end{array}\right) \cos nt + \left(\begin{array}{r}-a_2\\
-a_1\end{array}\right) \sin nt \colon a_1,a_2\in\bR\right\}.
\end{split}
\end{equation*}
It is easy to observe that
$A_1\colon \bR[1,(m,n)] \to E_n^1$
and $A_2\colon \bR[1,(-m,n)] \to E_n^2,$
given by
\begin{equation*}
\begin{split}
&A_1(a_1,a_2)=\left(\begin{array}{c}a_1\\ a_2\end{array}\right) \cos nt + \left(\begin{array}{r}-a_2\\ a_1\end{array}\right) \sin nt,\\
&A_2(a_1,a_2)=\left(\begin{array}{r}a_1\\ -a_2\end{array}\right) \cos nt + \left(\begin{array}{r}-a_2\\ -a_1\end{array}\right) \sin nt,
\end{split}
\end{equation*}
are $T^2$-equivariant isomorphisms.
Since $E_n=E_n^1 \oplus E_n^2$, we obtain that  $E_n\equiv_{T^2} \bR[1,(m,n)]\oplus \bR[1,(-m,n)].$ Analogously, in the case $m=0$ the isomorphism $A_0(a,b)=a \cos nt-b \sin nt$ gives the equivalence between  $\bR[1,(0,n)]$ and $E_n$. This proves the equivalence \eqref{eq:reprHn}. Analogously one can prove the equivalence \eqref{eq:reprH0}.
\end{proof}

\subsection{The Euler ring $U(T^2)$}
\label{sec:Euler}

Now we proceed to recall basic facts about the Euler ring of the torus $T^2$, needed in the proofs of our main results. For a more complete exposition and the general definition of the Euler ring $U(G)$ for any compact Lie group $G$ see for instance \cite{TomDieck1979}, \cite{TomDieck}. In our paper we restrict our attention to the case of the torus, following the notation given in \cite{GarRyb}.

Denote by $\sub(T^2)$ the set of all closed subgroups of $T^2$. For $H \in \sub(T^2)$ let $T^2/H$ be the orbit space of the $H$-action $H\times T^2\to T^2$ given by $(h,g)\mapsto gh^{-1}$ and $T^2/H^+$ be the disjoint union of $T^2/H$ and the base point $\{\star\}$. It is known that $T^2/H^+$ is a pointed $T^2$-CW-complex with only one 0-cell. Denote by $\chi_{T^2}(T^2/H^+)$ its equivariant Euler characteristic, see Definition 2.3 in \cite{GarRyb}. We define $U(T^2)$ as a $\bZ$-module generated by all elements $\chi_{T^2}(T^2/H^+)$ for $H \in \sub(T^2)$. Observe that this definition allows to understand $U(T^2)$ as  generated by all closed subgroups of $T^2$.

In the general case of any compact Lie group $G$ one can define the multiplication structure in $U(G)$ with the use of smash products of $G$-CW-complexes, see \cite{TomDieck1979}, \cite{TomDieck}. With this multiplication, the module $U(G)$ becomes a ring, called the Euler ring of the group $G$. In the case of $G=T^2$, we can use the result from \cite{GarRyb}, characterizing the multiplication in the Euler ring of a torus, instead of the general definition. Let us recall this result.
For $H_1, H_2\in \sub(T^2)$ put $H_0=H_1\cap H_2$. We have
\begin{equation}\label{eq:UT2multiplication}
\chi_{T^2}(T^2/H_1^+)\star\chi_{T^2}(T^2/H_2^+)=
\left\{\begin{array}{cl}
\chi_{T^2}(T^2/H_0^+), &\text{ if }\dim H_1+\dim H_2 =2+ \dim H_0, \\
\Theta & \text{ otherwise. }
\end{array}
\right.
\end{equation}
We denote by $\Theta$ the zero element and by $\bI=\chi_{T^2}(T^2/{T^2}^+)$ the identity in $U(T^2)$.

\begin{Remark}\label{rem:decomp}
By \eqref{eq:UT2multiplication}, for any $H_1, H_2\in \sub(T^2) \setminus \{T^2\}$ the product $\chi_{T^2}(T^2/H_1^+)\star\chi_{T^2}(T^2/H_2^+)$ is nonzero if and only if $\dim H_1=\dim H_2=1$ and $\dim H_1\cap H_2=0$. Let us reformulate it.
Put
$\ds U_i (T^2)=\bigoplus_{H\in\sub(T^2), \dim H=i} \bZ$ for $i=0,1,2$ and consider the decomposition
\[
U(T^2)= U_2(T^2)\oplus U_1 (T^2)\oplus U_0 (T^2).
\]
Then, the multiplication restricted to these subgroups has the following properties:
\begin{enumerate}[(i)]
\item $\star \colon  U_1 (T^2) \times  U_1 (T^2) \to  U_0 (T^2)$,
\item $\star \colon U_1 (T^2) \times  U_0 (T^2) \to  \{\Theta\}$,
\item $\star \colon  U_0 (T^2) \times  U_0 (T^2) \to  \{\Theta\}$.
\end{enumerate}
\end{Remark}



Let $(m,n) \in \bZ^2\setminus\{(0,0)\}$ and put
\begin{equation}\label{eq:defHmn}
H_{(m,n)}=\{(e^{i\phi_1}, e^{i\phi_2})\in T^2\colon e^{i(m\phi_1+n\phi_2)}=1\}.
\end{equation}
Note that if $v\in\bR[1,(m,n)]$ and $v\neq 0$, its isotropy group is $H_{(m,n)}$.

\begin{Remark}
Any subgroup $H \in \sub(T^2)\setminus\{T^2\}$ is of the form $H=H_{(m_1,n_1)} \cap \ldots \cap H_{(m_r,n_r)}$ for some $(m_1,n_1), \ldots, (m_r,n_r)\in \bZ^2\setminus\{(0,0)\}$. Indeed, it is known that for every closed subgroup $H$ of a compact Lie group $G$ there is a finite dimensional orthogonal representation $\bV$ and $v\in\bV$ such that the isotropy group of $v$ is equal to $H$, see Corollary 4.6.7 of \cite{DuisKolk}. Therefore, the characterization of closed subgroups of $T^2$ follows directly from the classification of finite dimensional orthogonal $T^2$-representations given in Corollary \ref{cor:repT}.
\end{Remark}

In the following lemmas we collect some properties of intersections of subgroups of $T^2$ and of the multiplication in $U(T^2)$, which we will need in the proofs of our main results.

\begin{Lemma}\label{lem:intersections1}
For any $k,n\in\bN$, $m\in\bZ$,
$$H_{(k, 0)}\cap H_{(m, n)}=\{(x,y)\in\bZ_k\times S^1\colon y^n=x^{-m}\}.$$
\end{Lemma}
\begin{proof} The lemma follows immediately from the formula \eqref{eq:defHmn}.
\end{proof}

\begin{Remark}\label{rem:twisted}
For a fixed $x\in \bZ_k$, the set $\{y\in S^1\colon y^n=x^{-m}\} $ is in fact the set of all complex $n$th roots of $x^{-m}$. In particular $H_{(k, 0)}\cap H_{(m, n)}$ contains $k\cdot n$ elements. Moreover, this intersection can be seen as a twisted subgroup of $\bZ_k\times S^1$ in the sense of \cite{BKS}. More precisely, if we consider the homomorphism $\phi_m\colon\bZ_k\to S^1$ given by $\phi_m(x)=x^{-m}$, then this intersection is the $n$-folded subgroup of $\bZ_k\times S^1$ twisted by the homomorphism $\phi_m$.
\end{Remark}

\begin{Lemma}\label{lem:intersections2}
If for some $n_1, n_2, k_1,k_2\in \bN$, $m_1,m_2\in\bZ$
$$H_{(k_1, 0)}\cap H_{(m_1, n_1)} =H_{(k_2, 0)}\cap H_{(m_2,n_2)},$$
then $n_1=n_2$, $k_1=k_2$ and $m_2=m_1+k_1 \cdot j$ for some $j\in \bZ$.
\end{Lemma}
\begin{proof}
From Lemma \ref{lem:intersections1} we get
$$H_{(k_i, 0)}\cap H_{(m_i, n_i)}=\{(x,y)\in\bZ_{k_i}\times S^1\colon y^{n_i}=x^{-m_i}\}$$
for $i=1,2$.
This yields $\bZ_{k_1}=\bZ_{k_2}$ and consequently $k_1=k_2$.
By Remark \ref{rem:twisted}, $k_1 n_1=k_2 n_2$ and hence $n_1=n_2$.
Putting now $k_0=k_1=k_2$ and $n_0=n_1=n_2$ and taking a pair
$(x,y)\in H_{(k_0, 0)}\cap H_{(m_i, n_0)}$ for $i=1,2$, we have $x\in\bZ_{k_0}$ and
$y^{n_0}=x^{-m_1}=x^{-m_2}$. Hence $x^{m_1-m_2}=x^{m_2-m_1}=1$ and $x\in \bZ_{|m_1-m_2|}$. Therefore $\bZ_{k_0}$ is a subgroup of $\bZ_{|m_1-m_2|}$ and consequently $k_0|(m_1-m_2)$, which proves the assertion.
\end{proof}

\begin{Lemma}\label{lem:intersections3}
Let
$$A=\sum_{k\in\bN} a_k \chi_{T^2}(T^2/H_{(k, 0)}^+),\ B=\sum_{(m,n)\in\bZ\times\bN} b_{(m,n)} \chi_{T^2}(T^2/H_{(m, n)}^+)\in U_1(T^2)\setminus\{\Theta\}.$$
If all the nonzero coefficients $b_{(m,n)}$ are of the same sign, then $A\star B \neq \Theta$.
\end{Lemma}

\begin{proof}
Fix $H_{(k_0, 0)}, H_{(m_0, n_0)}$ such that the corresponding coefficients $a_{k_0}, b_{(m_0,n_0)}$ are nonzero and put $H_0=H_{(k_0, 0)}\cap H_{(m_0, n_0)}$. Since $\dim H_0=0$ (see Lemma \ref{lem:intersections1}) and $\dim H_{(k_0, 0)}=\dim H_{(m_0, n_0)}=1$, using the formula \eqref{eq:UT2multiplication} we get
$$\chi_{T^2}(T^2/H_{(k_0, 0)}^+)\star \chi_{T^2}(T^2/H_{(m_0, n_0)}^+)=\chi_{T^2}(T^2/H_0^+) \neq \Theta.$$
Moreover,
\[
A\star B= c_{H_0}\chi_{T^2}(T^2/H_0^+)+\sum_{H\neq H_0} c_H \chi_{T^2}(T^2/H^+)\in U_0(T^2).
\]
Note that $c_{H_0}=\sum  a_{k} \cdot b_{(m,n)}$, where the sum is over all $k$, $m$, $n$ such that $H_{(k, 0)}\cap H_{(m, n)}=H_0$.
By Lemma \ref{lem:intersections2}, if
$H_{(k, 0)}\cap H_{(m, n)} =H_0,$
then  $n=n_0$, $k=k_0$ and $m=m_0+k_0\cdot j$ for some $j\in \bZ$. Hence
\begin{equation}\label{eq:ch0}
c_{H_0}=\sum_{j\in\bZ}  a_{k_0} \cdot b_{(m_0+k_0\cdot j,n_0)}=a_{k_0} \cdot \sum_{j\in\bZ}  b_{(m_0+k_0\cdot j,n_0)}.
\end{equation}
Since $b_{(m_0,n_0)}\neq 0$ and $b_{(m_0+k_0\cdot j,n_0)}$ have the same sign, $c_{H_0}\neq 0$. This finishes the proof.
\end{proof}

\begin{Corollary}\label{cor:intersections4}
Under the assumptions of Lemma \ref{lem:intersections3}, using the formula \eqref{eq:ch0}, we immediately obtain that if $b_{(m,n)}\leq 0$ for all $(m,n)$ and
\begin{enumerate}[(i)]
\item  $a_k\geq 0$ for all $k$, then all the nonzero coefficients of $A\star B$ are negative,
\item $a_k\leq 0$ for all $k$, then all the nonzero coefficients of $A\star B$ are positive.
\end{enumerate}
\end{Corollary}

\begin{Remark}\label{rem:Usone}
The main algebraic structure used in our paper is the Euler ring of the torus $T^2$. However, some of our assumptions are given in the terms of another ring, the Euler ring of the group $S^1$, denoted $U(S^1)$, which has a simpler structure. Similarly as in the case of a torus, the module $U(S^1)$ is generated by the $S^1$-equivariant Euler characteristics $\chi_{S^1}(S^1/H^+)$ with $H\in \sub(S^1)=\{S^1,\bZ_1,\bZ_2,\ldots\}$.
The elements of  $U(S^1)$ are of the form
$$ \alpha_0 \chi_{S^1}(S^1/S^{1+})+\alpha_1 \chi_{S^1}(S^1/\bZ_1^+)+\ldots +\alpha_k \chi_{S^1}(S^1/\bZ_k^+)+\ldots$$
where $\alpha_i \in\bZ$ and only finitely many of them are nonzero. More details about that ring, as well as about a more general ring of the $n$-dimensional torus $U(T^n)$, can be found for example in \cite{GarRyb}.
\end{Remark}

\subsection{The degree for equivariant gradient maps}
\label{degg}

Fix a compact Lie group $G$. Note that in our paper we are interested in the cases $G=T^2$ and $G=S^1$, but in this subsection, for the simplicity of notation, we describe the general case. Let $\bV$ be a finite dimensional, orthogonal representation of $G$, $\varphi \in C^1(\bV, \bR)$ a $G$-invariant function, and $\Omega \subset \bV$ an open, bounded, $G$-invariant set such that $\varphi$ is $\Omega$-admissible, i.e. $\partial \Omega \cap (\nabla \varphi)^{-1}(0) = \emptyset.$ For such a pair $(\varphi, \Omega)$ Gęba has defined in \cite{Geba} the degree for $G$-equivariant gradient maps, being an element of the Euler ring $U(G).$  We denote this degree by $\dg(\nabla \varphi, \Omega)$. This degree has analogs of the properties of the Brouwer degree, i.e. the properties of excision, additivity, linearisation and homotopy invariance. The precise formulations of these properties can be found in \cite{Geba}, \cite{Ryb2005milano}. Moreover in the proofs of our results we use the so called product formula, allowing to compute the degree of a product mapping. We formulate this property below.

\begin{Fact}\label{fact:multiplication} (Theorem 3.1 of \cite{GolRyb2013})
Let $\Omega_i \subset \bV_i$ be open, bounded and $G$-invariant subsets of finite dimensional $G$-representations $\bV_i$ and let $\varphi_i \in C^1(\bV_i, \bR)$ be $G$-invariant, $\Omega_i$-admissible functions for $i=1,2$. Then
\begin{equation}\label{eq:multiplication}
\dg((\nabla \varphi_1, \nabla \varphi_2), \Omega_1 \times \Omega_2)=\dg(\nabla \varphi_1, \Omega_1) \star \dg(\nabla \varphi_2, \Omega_2).
\end{equation}
\end{Fact}

Using the description of the Euler ring given in Section \ref{sec:Euler},  we write $$\degttwo(\nabla \varphi, \Omega) = \sum_{H\in\sub(T^2)} n_H \cdot \chi_{T^2}(T^2/H^+),$$ and analogously, $$\degsone(\nabla \varphi, \Omega)=\sum_{H\in\sub(S^1)} n_H \cdot \chi_{S^1}(S^1/H^+).$$ Taking into account Remark \ref{rem:Usone}, we also use the notation
$$\degsone(\nabla \varphi, \Omega)=\alpha_0 \chi_{S^1}(S^1/S^{1+})+\alpha_1 \chi_{S^1}(S^1/\bZ_1^+)+\ldots +\alpha_k \chi_{S^1}(S^1/\bZ_k^+)+\ldots.$$

In our results we consider mainly $G=T^2$ and mappings acting on the representations $\bH_n$ given by the formula \eqref{eq:defHn}. The special case of such a  representation is obtained for $n=0$. In this situation we have the following relation of the degrees on the $T^2$-representation $\bH_0$ and on the $S^1$-representation $\bR^N$:

\begin{Remark}\label{lem:coefficients}
Let $\bR^N$ be an $S^1$-representation, $\Omega \subset \bR^N$ an $S^1$-invariant subset and $\varphi \in C^1(\bV, \bR)$ an $S^1$-invariant function such that $\partial \Omega \cap (\nabla \varphi)^{-1}(0)=\emptyset$. Consider the $T^2$-representation $\bH_0$ given by \eqref{eq:defHn}. Without loss of generality we can consider $\Omega$ as a  $T^2$-invariant subset of $\bH_0$ and $\varphi \in C^1(\bH_0, \bR)$ as a $T^2$-invariant function. From Remark \ref{rem:H0} and the definition of the degree for equivariant gradient maps it follows that if
$$\degsone(\nabla \varphi, \Omega) = \alpha_0 \chi_{S^1}(S^1/S^{1+})+\alpha_1 \chi_{S^1}(S^1/\bZ_1^+)+\ldots +\alpha_k \chi_{S^1}(S^1/\bZ_k^+)+\ldots$$ then
$$\degttwo(\nabla \varphi, \Omega) = \alpha_0 \chittwo(T^2/T^{2+})+\alpha_1 \chittwo(T^2/H_{(1,0)}^+)+\ldots +\alpha_k \chittwo(T^2/H_{(k,0)}^+)+\ldots.$$

\end{Remark}
We are particularly interested in computing the degree in the case of an isomorphism. It is known, see the formula (4.1) of \cite{Ryb2005milano}, that such computations can be reduced to computing the degree of a minus identity. On the other hand, for $G$ being a torus, using Theorem 4.3 of \cite{Ryb2005milano} and Theorem 3.2 of \cite{GarRyb}, we can obtain an explicit formula for such a degree. In particular, for $G=T^2$, if $\bV=\bR[k_0,(0,0)]\oplus \bR[k_1,(m_1,n_1)]\oplus\ldots \oplus \bR[k_r,(m_r,n_r)]$ we have
\begin{equation}\label{eq:degminusid}
\begin{split}
\degttwo(-Id, B_{\alpha}(\bV))=&(-1)^{k_0}\left(\chi_{T^2}(T^2/T^{2+})-\sum_{i=1}^r  k_i \cdot \chi_{T^2}(T^2/H_{(m_i,n_i)}^+)\right)+\\&+\sum_{H \in \{H  \in \sub(T^2)\colon \dim H =0\}} k_{H}\cdot \chi_{T^2}(T^2/H^+).
\end{split}
\end{equation}

\begin{Remark}
Note that in the case of the irreducible representation $\bR[1,(m,n)]$ for $(m,n) \neq(0,0)$, the above formula can be computed directly from the definition of the degree for equivariant gradient maps. In this case we obtain
\begin{equation}\label{eq:degminusid2}
\degttwo(-Id, B_{\alpha}(\bR[1,(m,n)]))=\chi_{T^2}(T^2/T^{2+})-\chi_{T^2}(T^2/H_{(m,n)}^+).
\end{equation}
\end{Remark}

The degree for equivariant gradient maps has its generalization in infinite-dimensional case. Let $\bH$ be an infinite-dimensional, separable Hilbert space which is an orthogonal representation of $G$ and $\Phi \in C^1(\bH, \bR)$ be a $G$-invariant functional such that
\begin{equation}\label{eq:formofPhi}
\nabla \Phi(u)=u-\nabla \eta(u),
\end{equation}
where $\nabla \eta \colon \bH \to \bH$ is a $G$-equivariant completely continuous operator. Moreover, let $\Omega \subset \bH$ be an open, bounded and a $G$-invariant subset such that  $\partial \Omega \cap (\nabla \Phi)^{-1}(0) = \emptyset.$ For such a pair $(\Phi, \Omega)$, one can define the infinite-dimensional version of the degree for equivariant gradient maps, see \cite{Ryb2005milano}. For simplicity of notation, we denote this degree by the same symbol as in finite dimensional case, i.e. $\dg(\nabla \Phi, \Omega)$. This degree has properties of excision, additivity, linearization and homotopy invariance, see Theorem 4.5 of \cite{Ryb2005milano}.

\begin{Remark}\label{rem:multiplication}
The product formula stated in Fact \ref{fact:multiplication} holds also when one or both of the mappings $\varphi_i$ in formula \eqref{eq:multiplication} are of the form \eqref{eq:formofPhi}. This is a direct consequence of Fact \ref{fact:multiplication} and the definition of the infinite-dimensional version of the degree for equivariant gradient maps.
\end{Remark}

\section{Periodic solutions of Newtonian systems}\label{sec:main}

Now we turn our attention to the main result of the article on the existence of connected sets of solutions of some Newtonian systems. We start with formulating the main theorem and then, throughout the remainder of the section, we give its proof.

\subsection{Formulation of the main result} \label{sec:formulation}
Our aim is to study properties of continua (i.e. closed, connected sets) of $2\pi$-periodic solutions of a family of autonomous Newtonian systems of the form:
\begin{equation}\label{eq:newtonian}
\ddot{u}(t)=-\lambda^2 U'(u(t)),
\end{equation}
where  $\lambda >0$. 
We assume that
\begin{enumerate}
\item[(a1)] $\bR^N$ is an orthogonal representation of the group $S^1$,  $\Omega\subset\bR^N$ is an open, $S^1$-invariant subset and the potential  $U \in C^2(\Omega, \bR)$ is $S^1$-invariant. Moreover,
there exist $C>0$ and $s\in[1,+\infty)$ such that $|U'(u)|\leq C(1+|u|^{s})$,
\item[(a2)] $(U')^{-1}(0)$ is finite,
\item[(a3)]  there exists $u_0\in (U')^{-1}(0)$ such that
\begin{enumerate}
\item[(a3.1)] $\sigma(U''(u_0))\cap (0,+\infty)\neq \emptyset$,
\item[(a3.2)] $\deg^{\nabla}_{\sone}(-U',B_{\delta}(u_0,\bR^N))\neq \Theta\in U(S^1)$ for $\delta>0$ such that $B_{\delta}(u_0,\bR^N)\cap(U')^{-1}(0) =\{u_0\}.$
\end{enumerate}
\end{enumerate}

%

Below we formulate the main result of our paper. We start with introducing the notation.
Let $C_{2 \pi}([0,2\pi],\Omega) =\{u \in C([0,2\pi], \Omega)\colon u(0)=u(2\pi)\}.$ It is clear that $C_{2 \pi}([0,2\pi],\Omega) $ is an open subset of $C_{2\pi}([0,2\pi],\bR^N)$. We consider $2\pi$-periodic solutions of the system \eqref{eq:newtonian} as elements  $(u, \lambda) \in C_{2 \pi}([0,2\pi],\Omega) \times (0, +\infty).$
Let $(U')^{-1}(0)=\{u_0, \ldots, u_q\}$. Identifying the points $u_0, \ldots, u_q\in \Omega$ with constant functions in $C_{2 \pi}([0,2\pi],\Omega) $, we define two subsets $\cT, \cN\cT\subset C_{2 \pi}([0,2\pi],\Omega)  \times (0,+\infty)$ as follows
\begin{equation*}
\begin{split}
&\cT=\{u_0, \ldots, u_q\}\times(0, +\infty),\\
&\cN\cT=\{(u,\lambda)\colon (u,\lambda) \text{ is a nonstationary } 2\pi\text{-periodic solution of the system \eqref{eq:newtonian}}\}.
\end{split}
\end{equation*}

The elements of $\cT$ are usually called trivial solutions of the system \eqref{eq:newtonian} and elements of $\cN\cT$ nontrivial ones. In other words, in this terminology the stationary solutions are treated as the trivial ones.

Assume that $\lambda_0\in (0, +\infty)$ is such that $(u_0, \lambda_0) \in \cl(\cN\cT)$ and denote by $\cC(u_{0}, \lambda_0)$ the continuum of $\cl(\cN\cT)$ containing $(u_{0}, \lambda_0)$, where the closure of $\cN\cT$ is taken in $C_{2 \pi}([0,2\pi],\Omega)  \times (0, +\infty).$

Observe that the assumption (a3.1) means that there exists at least one positive eigenvalue of $U''(u_0)$.   Put $\Lambda(u_0) = \{\frac{k}{\beta}\colon k \in \bN, \beta \in (0,+\infty) \text{ and } \beta^2 \in \sigma(U''(u_0))\}.$

We claim that for $\lambda_0 \in \Lambda(u_0)$, under the above assumptions, from the stationary solution $(u_0, \lambda_0)$ emanates a connected family of nonstationary $2\pi$-periodic solutions of the system \eqref{eq:newtonian}. More precisely, we are going to prove the following theorem.

\begin{Theorem}
\label{thm:main}
Consider the system \eqref{eq:newtonian} with the potential $U$ and $u_0 \in (U')^{-1}(0)$ satisfying the assumptions (a1)-(a3). Then for $\lambda_0 \in \Lambda(u_0)$  we have $(u_0, \lambda_0) \in \cT \cap \cl(\cN\cT)$. Moreover $\cC(u_0, \lambda_0) \neq \{(u_0, \lambda_0)\}$ and either this set is not compact in $C_{2 \pi}([0,2\pi],\Omega)  \times (0, +\infty)$  or it is compact and $\cC(u_0, \lambda_0) \cap (\cT \setminus \{(u_0, \lambda_0)\}) \neq \emptyset.$
\end{Theorem}

The proof of this theorem is given in Sections \ref{sec:var}-\ref{sec:proofmain}. The idea of this proof is as follows. We consider periodic solutions of the system \eqref{eq:newtonian} as critical points of some associated functional. The definition and properties of this functional are given in Section \ref{sec:var}. To study the set of critical points of such a functional we use the theory of global bifurcation. In Section \ref{sec:glob} we give the definition of this phenomenon and recall the necessary condition for its occurrence. We also define a bifurcation index being an element of $U(T^2)$. Using an equivariant global bifurcation theorem we show that to prove Theorem \ref{thm:main} it is enough to prove the nontriviality of such an index. This final step of the proof is done in Section \ref{sec:proofmain}.

\subsection{Variational setting}
\label{sec:var}

Define
$$\h^1_{2\pi} = \{u\colon [0,2\pi] \rightarrow \bR^N\colon  u \text{ is absolutely continuous, } u(0)=u(2\pi), \dot u \in L^2([0,2\pi],\bR^N)\}.$$
It is well-known that $\h^1_{2\pi}$ with the inner product
$$\ds \langle u,v\rangle_{\h^1_{2\pi}} = \int_0^{2\pi} (\dot u(t), \dot v(t)) + (u(t),v(t)) \; dt$$
is a separable Hilbert space.
Consider an open subset $\h^1_{2\pi}(\Omega)\subset\h^1_{2\pi}$ defined by $\h^1_{2\pi}(\Omega)=\{u\in \h^1_{2\pi}\colon u([0,2\pi])\subset\Omega\}$.
Define a $C^2$-functional $\Phi \colon\h^1_{2\pi}(\Omega) \times (0,+\infty) \to \bR$ by
\begin{equation}\label{eq:Phi}
\Phi(u,\lambda) = \int_0^{2\pi} \left( \frac{1}{2} \| \dot u(t) \|^2 -  \lambda^2U(u(t)) \right) \; dt.
\end{equation}

It is known that the gradient $\nabla \Phi$ is of the form of a completely continuous perturbation of the identity, see \cite{Rab}. Moreover, its zeros are in one-to-one correspondence with $2\pi$-periodic solutions of the system \eqref{eq:newtonian}.

Since we assume the $S^1$-symmetry of the system, the associated functional $\Phi$ is $T^2$-symmetric.
Namely, recall that, by (a1), $\bR^N$ is an $S^1$-representation which, by Corollary \ref{cor:s1repr}, can be assumed to be equivalent to
\begin{equation}\label{eq:sotworepr}
\bV=\bR[k_0,0] \oplus \bR[k_1, m_1] \oplus \ldots \oplus \bR[k_r,m_r],
\end{equation}
for some $k_1, \ldots, k_r, m_1, \ldots, m_r \in \bN, k_0 \in \bN\cup \{0\}.$
Therefore the space $\hone$ is an orthogonal representation of the group $T^2=S^1 \times S^1$ with the action given by
\begin{equation}\label{eq:action}
(e^{i \phi_1}, e^{i\phi_2})(u)(t)=\rho_{\bV}(e^{i\phi_1}) \left(u(t+\phi_2)\right).
\end{equation}
It easily follows that $\h^1_{2\pi}(\Omega)$ is a $T^2$-invariant subset of $\h^1_{2\pi}$ and $\Phi$ is $T^2$-invariant.

 It is known that
 $$\hone = \overline{\h_0 \oplus \bigoplus_{n=1}^{\infty} \h_n},$$
where  $\h_n =\{u(t)=a\cos nt+b\sin nt\colon a,b\in\bR^N\}$, for $n \geq 0$. Then, from the formula \eqref{eq:sotworepr} and Lemma \ref{Lem:Hn},
\begin{equation}\label{eq:opisHn}
\begin{split}
&\bH_0 \equiv_{T^2} \bR[k_0,(0,0)] \oplus \bR[k_1,(m_1,0)] \oplus \ldots \oplus \bR[k_r,(m_r,0)],\\
&\bH_n \equiv_{T^2} \bR[k_0,(0,n)]\oplus \bigoplus_{j=1}^{r}\big(\bR[k_j,(m_j,n)] \oplus  \bR[k_j,(-m_j,n)]\big).
\end{split}
\end{equation}

Denote by $\alpha_1, \ldots, \alpha_q$ the distinct eigenvalues of $A=U''(u_0)$ and let  $V_A(\alpha_1), \ldots, V_A(\alpha_q)$ be the corresponding eigenspaces. Thus  $\bR^N=V_A(\alpha_1) \oplus \ldots \oplus V_A(\alpha_q)$. Since $U$ is $S^1$-invariant, all these eigenspaces are $S^1$-representations.

As $\bR^N$ decomposes to a direct sum of $V_A(\alpha_i)$, we decompose $\bH_n$ to $\cV_A^n(\alpha_i)$, where
$\cV_A^n(\alpha_i)=\{a_n \cos nt +b_n \sin nt\colon a_n, b_n \in V_A(\alpha_i)\} \subset \bH_n.$ Taking into consideration the formula \eqref{eq:action}, the subspaces $\cV_A^n(\alpha_i)$ are $T^2$-subrepresentations of $\bH_n$.
More precisely, by the formula \eqref{eq:sotworepr}, there are $k_0(\alpha_i)\in \bN \cup \{0\}$, $k_1(\alpha_i)\ldots, k_{s_i}(\alpha_i) \in \bN$, $m_1(\alpha_i),\ldots, m_{s_i}(\alpha_i)\in \{m_1,\ldots,m_r\}$ such that
\begin{equation}\label{eq:opisVA0}
V_A(\alpha_i)\equiv_{S^1} \bR[k_0(\alpha_i),0]\oplus\bigoplus_{j=1}^{s_i} \bR[k_j(\alpha_i),m_j(\alpha_i)].
\end{equation}
Therefore, by Lemma \ref{Lem:Hn}, each $\cV_A^n(\alpha_i)$ can be decomposed as a $T^2$-representation as:
\begin{equation}\label{eq:opisVA}
\cV_A^n(\alpha_i)\equiv_{T^2} \bR[k_0(\alpha_i),(0,n)]\oplus\bigoplus_{j=1}^{s_i} (\bR[k_j(\alpha_i),(m_j(\alpha_i),n)] \oplus \bR[k_j(\alpha_i),(-m_j(\alpha_i),n)]).
\end{equation}

In our computations we will use the formulas characterizing the actions of the linear operators $\nabla^2\Phi(u_0,\lambda)$ on the representations $\bH_n$.
From Lemma 5.1.1 in \cite{FRR} it follows that for $u(t) = a_n \cos nt+b_n \sin nt \in \bH_n$, where $n\geq 0$, we have
\begin{equation}\label{eq:hesjan}
\nabla^2\Phi(u_0,\lambda)(u)(t)=\left(\frac{n^2}{n^2+1}Id - \frac{\lambda^2}{n^2+1}A\right) a_n \cos nt+ \left(\frac{n^2}{n^2+1}Id - \frac{\lambda^2}{n^2+1}A\right)b_n \sin nt.
\end{equation}

\subsection{Global bifurcation}\label{sec:glob}
Recall that $2\pi$-periodic solutions of the system \eqref{eq:newtonian} are in one-to-one correspondence with critical points of the functional $\Phi$ given by the formula \eqref{eq:Phi}. In other words, instead of solutions of the system \eqref{eq:newtonian} we consider
solutions of the equation
\begin{equation}\label{eq:crit}
\nabla_u \Phi(u, \lambda)=0.
\end{equation}


Consider the sets $\cT, \cN\cT$ defined in Section \ref{sec:formulation}. Observe that $\cT=\{u_0, \ldots, u_q\} \times (0, +\infty)\subset \h^1_{2\pi}(\Omega)\times (0, +\infty)$ and
$$\cN\cT=\{(u, \lambda) \in \hone(\Omega) \times (0, +\infty) \colon \nabla_u\Phi(u,\lambda)=0, (u,\lambda)\notin \cT\}.$$
In this section we want to study the bifurcation of
nontrivial solutions of the equation \eqref{eq:crit} from the set of trivial ones. For this purpose we formulate the following definition:

\begin{Definition}\label{def:globbif}
 A point $(u_{i_0},\lambda_0)\in \cT$ is called a point of a global bifurcation of nontrivial solutions of the equation \eqref{eq:crit} if there is a closed connected component $\cC(u_{i_0},\lambda_0)\subset \hone(\Omega) \times (0, +\infty)$ of $\cl(\cN\cT)$ containing $(u_{i_0},\lambda_0)$ and such that either $\cC(u_{i_0},\lambda_0)$ is not compact or $\cC(u_{i_0},\lambda_0)$ is compact and $\cC(u_{i_0},\lambda_0)\cap  (\cT\setminus \{(u_{i_0},\lambda_0)\}) \neq \emptyset$.
\end{Definition}

Below we formulate known conditions for a global bifurcation. We start with the necessary one. To formulate this condition we observe that a bifurcation from $(u_i, \lambda)$ can occur only when this point is an accumulation point of nontrivial solutions of \eqref{eq:crit}. Put $$\Lambda(u_i) = \left\{\frac{k}{\beta}\colon k \in \bN, \beta \in (0,+\infty) \text{ and } \beta^2 \in \sigma(U''(u_i))\right\}$$ for $i=0,\ldots, q$. Then, from Theorem 3.2.1 of \cite{PRS2} we have:

\begin{Fact}\label{fact:necessary}
If $(u_i, \lambda) \in \cl(\cN\cT)$, then $\lambda \in \Lambda(u_i).$
\end{Fact}

As a consequence of the above fact we obtain the following necessary condition for a global bifurcation:
\begin{Corollary}
 If a global bifurcation occurs from $(u_i, \lambda)$, then $\lambda \in \Lambda(u_i).$
\end{Corollary}

To obtain the sufficient condition we are going to apply the degree for $T^2$-equivariant gradient maps. Let $u_{i_0}\in (U')^{-1}(0)$ be such that $\Lambda(u_{i_0})\neq \emptyset$. Since the set $\Lambda(u_{i_0})$ does not have finite accumulation points,  there exist $\lambda_0 \in \Lambda(u_{i_0})$ and $\varepsilon>0$ such that $[\lambda_0 - \varepsilon, \lambda_0+\varepsilon]\cap \Lambda(u_{i_0}) = \{\lambda_0\}.$
Using Fact \ref{fact:necessary} we observe that the condition $\lambda_0 \pm \varepsilon \notin \Lambda(u_{i_0})$ implies that there is $\delta>0$ such that  $\nabla_u \Phi(\cdot, \lambda_0 \pm \varepsilon)^{-1}(0)\cap B_{\delta}(u_{i_0}, \bH^1_{2\pi}(\Omega))=\{u_{i_0}\}$. Hence the degrees $\degttwo(\nabla_u \Phi(\cdot, \lambda_0 \pm \varepsilon), B_{\delta}(u_{i_0}, \bH^1_{2\pi}(\Omega)))$ are well-defined. For such  $\delta$, we have   $B_{\delta}(u_{i_0}, \bH^1_{2\pi}(\Omega))=B_{\delta}(u_{i_0}, \bH^1_{2\pi}).$ Therefore we can define the bifurcation index $\bif_{T^2} (u_{i_0},\lambda_0)\in U(T^2)$ by:
\begin{equation}\label{eq:indbif}
\begin{split}
&\bif_{T^2} (u_{i_0},\lambda_0)=\\
&=\mathrm{deg}^{\nabla}_{T^2}(\nabla_u \Phi(\cdot,\lambda_0+\varepsilon), B_{\delta}(u_{i_0},\hone))-\mathrm{deg}^{\nabla}_{T^2}(\nabla_u \Phi(\cdot,\lambda_0-\varepsilon), B_{\delta}(u_{i_0},\hone)).
\end{split}
\end{equation}

Nontriviality of this bifurcation index implies the global bifurcation. More precisely, we have the following version of the Rabinowitz global bifurcation theorem, see Theorem 4.9 of \cite{Ryb2005milano}:
\begin{Theorem}\label{thm:biF_in_H}
If $\bif_{T^2} (u_{i_0},\lambda_0)\neq \Theta\in U(T^2)$, then a global bifurcation of solutions of the equation \eqref{eq:crit} occurs from $(u_{i_0},\lambda_0)$.
\end{Theorem}

Note that the bifurcation result given in Theorem \ref{thm:biF_in_H} is obtained in the space $\bH^1_{2\pi}(\Omega) \times (0, +\infty)$. To increase its readability, we reformulate it in the space of continuous functions.

\begin{Corollary}\label{thm:bif_in_cont}
If $\bif_{T^2} (u_{i_0},\lambda_0)\neq \Theta\in U(T^2)$, then the global bifurcation of solutions of the system \eqref{eq:newtonian} occurring from $(u_{i_0},\lambda_0)$, takes place in $C_{2 \pi}([0,2\pi],\Omega) \times (0,+\infty)$.
More precisely, there is a closed connected component $\cC(u_{i_0},\lambda_0)\subset C_{2 \pi}([0,2\pi],\Omega) \times (0,+\infty)$ of $\cl(\cN\cT)$ containing $(u_{i_0},\lambda_0)$ and such that either $\cC(u_{i_0},\lambda_0)$ is not compact in $C_{2 \pi}([0,2\pi],\Omega) \times (0,+\infty)$ or $\cC(u_{i_0},\lambda_0)$ is compact and $\cC(u_{i_0},\lambda_0)\cap  (\cT\setminus \{(u_{i_0},\lambda_0)\}) \neq \emptyset$.
\end{Corollary}

\begin{proof}


The fact that closed and compact sets in $\cN\cT\cup\cT$ are the same in the norms $\|\cdot\|_{\infty}$ and  $\|\cdot\|_{\bH^1_{2\pi}}$ follows from the observation that the convergence in $\cN\cT\cup\cT$ in $\|\cdot\|_{\infty}$  is equivalent to the convergence in  $\|\cdot\|_{\bH^1_{2\pi}}$.

To prove this equivalence, recall first  that $\bH^1_{2\pi}(\Omega)$ is continuously embedded in $C_{2 \pi}([0,2\pi],\Omega)$ (see Proposition 1.1 of \cite{MawWil}).
From this we immediately obtain that the convergence in $\bH^1_{2\pi}(\Omega) \times (0,+\infty)$ implies the convergence in $C_{2 \pi}([0,2\pi],\Omega)  \times (0,+\infty)$.


To show the opposite, suppose that $\{(u_n,\lambda_n)\}\subset\cN\cT\cup\cT\subset C_{2 \pi}([0,2\pi],\Omega)  \times (0,+\infty)$ is convergent to $(\bar{u},\bar{\lambda})$. Since the limit of a uniformly convergent sequence of solutions of the system \eqref{eq:newtonian} is also a solution, $(\bar{u},\bar{\lambda})\in \bH^1_{2\pi}(\Omega) \times (0,+\infty)\cap (\cN\cT\cup\cT)$.


What is left is to show that the sequence $\{(u_n,\lambda_n)\}$ also converges to $(\bar{u},\bar{\lambda})$ in $\bH^1_{2\pi}(\Omega) \times (0,+\infty)$.
A standard argument shows that
\begin{equation}\label{eq:convinC}
\|u_n-\bar{u}\|_{L^2}\leq 2 \pi \|u_n-\bar{u}\|_{\infty}.
\end{equation}
Moreover, since $u_n$, $\bar{u}$ are weak solutions of the system \eqref{eq:newtonian}, we get
\begin{equation*}
\begin{split}
  \|\dot u_n-\dot{\bar{u}}\|_{L^2}^2=&
\int_0^{2\pi} \left(\dot u_n(t)-\dot{\bar{u}}(t),\dot u_n(t)-\dot{\bar{u}}(t)\right)  dt\\
 = &\int_0^{2\pi}\left(u_n(t)-\bar{u}(t), \lambda_n^2U'( u_n(t))- \bar{\lambda}^2U'(\bar{u}(t))\right) dt  \\ \leq&
\int_0^{2\pi}  | u_n(t)-\bar{u}(t)| \cdot |\lambda_n^2U'( u_n(t))-\bar{\lambda}^2 U'(\bar{u}(t))| dt \\
\leq &  \|u_n-\bar{u}\|_{\infty} \int_0^{2\pi} |U'( u_n(t))|\cdot|\lambda_n^2-\bar{\lambda}^2|+\bar{\lambda}^2 |U'( u_n(t))-U'(\bar{u}(t))| dt.
\end{split}
\end{equation*}
Since $U'$ is continuous and $\{u_n\}$ is uniformly convergent to $\bar{u}$, and in particular the sequence $\{(|U'( u_n(t))|)\}$ is bounded, the integrand converges to 0. Combining this with the equation \eqref{eq:convinC}, we obtain $(u_n,\lambda_n)\to (\bar{u},\bar{\lambda})$ in $\bH^1_{2\pi}(\Omega) \times (0,+\infty)$. This finishes the proof.
\end{proof}

\begin{Remark}
Suppose that $\cC(u_{i_0},\lambda_0)\subset \hone(\Omega) \times (0, +\infty)$ is a continuum of solutions of the equation \eqref{eq:crit} bifurcating from $(u_{i_0},\lambda_0)$ and
assume that it is not compact. Reasoning as in Step 2 of the proof of Theorem 3.3 of \cite{danryb}, one can show that $\cC(u_{i_0},\lambda_0)$ cannot reach the level $\lambda=0$. More precisely, if $(u_n,\lambda_n)\in\hone(\Omega) \times (0, +\infty)$ is a uniformly bounded sequence of solutions, then  $\lambda_n$ cannot converge to $0$ as $n\to\infty$.
\end{Remark}

\subsection{Proof of Theorem \ref{thm:main}}\label{sec:proofmain}

In this subsection we give the proof of Theorem \ref{thm:main}.

First of all observe that the assumption (a3.1) implies that $\Lambda(u_0) \neq \emptyset.$ Fix $\lambda_0 \in \Lambda(u_0)$ 
and $\varepsilon >0$ such that $[\lambda_0 - \varepsilon, \lambda_0+ \varepsilon] \cap \Lambda(u_0) = \{\lambda_0\}.$ Therefore, the bifurcation index $\bifttwo(u_0, \lambda_0)\in U(T^2)$ given by the formula \eqref{eq:indbif} is well-defined.
From Corollary \ref{thm:bif_in_cont} it follows that to prove the assertion it is enough to show the nontriviality of this index, i.e. to show
\begin{equation*}
\bif_{T^2} (u_{0},\lambda_0)
=\mathrm{deg}^{\nabla}_{T^2}(\nabla \Phi_+, B_{\delta}(u_0,\hone))-\mathrm{deg}^{\nabla}_{T^2}(\nabla \Phi_-, B_{\delta}(u_0,\hone))
 \neq \Theta\in U(T^2),
 \end{equation*}
where $\Phi_{\pm}(u)=\Phi(u, \lambda_0 \pm \varepsilon)$.

To that end we are going to compute these degrees. Without loss of generality, we assume that $u_0=0$. In particular, all the considered balls will be centered at $u_0=0$.

The critical point $u_{0}$ of $\Phi_{\pm}$ does not have to be nondegenerate. That is why to compute the degrees $\degttwo(\nabla \Phi_{\pm}, B_{\delta}( \hone ))$ we will apply Splitting Lemma, see Lemma 3.2 of \cite{FRR}. 
Put $\cL_{\pm}=\nabla^2\Phi_{\pm}(u_0) \colon \bH^1_{2\pi} \to \bH^1_{2\pi}.$
Since $\lambda_0 \pm \varepsilon \notin \Lambda(u_0)$, by the formula \eqref{eq:hesjan} we obtain $\ker(\cL_+)=\ker(\cL_-)\subset \bH_0$ and consequently, by self-adjointness of $\cL_{\pm}$, $\im(\cL _+)=\im(\cL_-).$ We denote such kernel and image by $\cN$ and $\cR$, respectively.
 From Splitting Lemma we obtain the existence of $\gamma>0$ and $T^2$-equivariant, $B_{\gamma}(\cN) \times B_{\gamma}(\cR)$-admissible homotopies connecting operators $\nabla \Phi_{\pm}$ with product mappings  $(\nabla \varphi_{\pm},(\cL_{\pm})_{|\cR})\colon B_{\gamma}(\cN)\times\cR \to \cN\times\cR$, where $\nabla \varphi_{\pm}$ are some $T^2$-equivariant gradient maps. Without loss of generality we assume that $\gamma = \delta$. The homotopy invariance property of the degree implies that
$$\mathrm{deg}^{\nabla}_{T^2}(\nabla \Phi_{\pm}, B_{\delta}(\hone)) = \mathrm{deg}^{\nabla}_{T^2}((\nabla \varphi_{\pm}, (\cL_{\pm})_{|\cR}), B_{\delta}(\cN) \times B_{\delta}(\cR)).$$
Put $\cR_1=\bH_0\cap\cN^{\bot}$ and $\cR_2 = \cR \cap \cR_1^{\bot}$. Since $(\cL_{\pm})_{|\cR}$ are isomorphisms, and $(\nabla \varphi_{\pm}, (\cL_{\pm})_{|\cR})$ are $B_{\delta}(\cN) \times B_{\delta}(\cR)$-admissible, it is easy to observe that $(\nabla \varphi_{\pm},(\cL_{\pm})_{|\cR_1})$ is $ B_{\delta}(\bH_0)$-admissible and therefore the product formula for the degree for equivariant gradient maps, see Remark \ref{rem:multiplication}, implies that
\begin{equation}\label{eq:splitted}
\mathrm{deg}^{\nabla}_{T^2}(\nabla \Phi_{\pm}, B_{\delta}(\bH^1_{2\pi})) =\mathrm{deg}^{\nabla}_{T^2}((\nabla \varphi_{\pm},(\cL_{\pm})_{|\cR_1}), B_{\delta}(\bH_0)) \star\mathrm{deg}^{\nabla}_{T^2}((\cL_{\pm})_{|\cR_2}, B_{\delta}(\cR_2)).
\end{equation}
 It is known that
$\nabla ({\Phi_{\pm}}_{|\bH_0})=(\nabla {\Phi_{\pm}})_{|\bH_0}$.
Using again the homotopy given by the splitting lemma, restricted to $\bH_0$, we obtain
$$\mathrm{deg}^{\nabla}_{T^2}((\nabla \varphi_{\pm},(\cL_{\pm})_{|\cR_1}), B_{\delta}(\bH_0)) = \degttwo((\nabla {\Phi_{\pm}})_{|\bH_0}, B_{\delta}( \bH_0)).$$

From the definition of $\Phi$, we obtain $(\nabla {\Phi_{\pm}})_{|\bH_0}(u)(t)=-(\lambda_0 \pm \varepsilon)^2 U'(u(t))$. Moreover, by the assumption (a3.2), there exist numbers $\mathfrak{n}_i \in\bZ$, $i=0,1,2,\ldots,$ and at least one of them is nonzero, such that
\begin{equation}\label{eq:froma32}
\deg^{\nabla}_{S^1}(-U',B_{\delta}(\bR^N))=
\mathfrak{n}_0 \cdot \chi_{S^1}(S^1/S^{1 +}) +\sum_{i=1}^{\infty} \mathfrak{n}_{i} \cdot \chi_{S^1}(S^1/\bZ_{i}^+) \in U(S^1)
\end{equation}
and, consequently, by Remark \ref{lem:coefficients},
\begin{equation}\label{eq:deg1}
\degttwo((\nabla {\Phi_{\pm}})_{|\bH_0}, B_{\delta}( \bH_0)) = \mathfrak{n}_0 \cdot \chi_{T^2}(T^2/T^{2 +}) +\sum_{i=1}^{\infty} \mathfrak{n}_{i} \cdot \chi_{T^2}(T^2/H_{(i,0)}^+)\in U_2(T^2)\oplus U_1(T^2).
\end{equation}


To compute the latter factor of \eqref{eq:splitted} we use the properties of the degree of an isomorphism. It is known, see Theorem 4.7 of \cite{Ryb2005milano}, that
$$\mathrm{deg}^{\nabla}_{T^2}((\cL_{\pm})_{|\cR_2}, B_{\delta}(\cR_2))= \mathrm{deg}^{\nabla}_{T^2}(-Id, B_{\delta}(\bV^-_{\pm})),$$
where $\bV_{\pm}^-$ is the direct sum of eigenspaces corresponding to negative eigenvalues of $\cL_{\pm}$.
Note that from the spectral properties of completely continuous perturbations of the identity it follows that the $T^2$-representations $\bV_{\pm}^-$ are finite-dimensional.

The precise description of the spaces $\bV_{\pm}^-$ can be obtained from the  characterization of the action of the linear operators $\cL_{\pm}$ on the representations $\bH_n$, see the formula \eqref{eq:hesjan}.
This implies that the eigenvalues of $\cL_{\pm}$ on $\bH_n$ are of the form $\frac{n^2-(\lambda_0 \pm \varepsilon)^2 \alpha}{n^2+1}$, for $\alpha$ being the eigenvalues of $A=U''(u_0)$, and the corresponding eigenspaces are $\cV_A^n(\alpha).$ Therefore
\begin{equation}\label{eq:v-+}
\bV_{\pm}^-=\bigoplus_{n=1}^{\infty} \bigoplus_{\alpha>\frac{n^2}{(\lambda_0 \pm \varepsilon)^2}} \cV_A^n (\alpha).
\end{equation}
Obviously, for $(\lambda_0 \pm \varepsilon)$ and $\alpha$ fixed, only finite number of $n$ satisfies $n^2-(\lambda_0 \pm \varepsilon)^2 \alpha<0$.
Hence there exists $\tilde{n}$ such that $\bV_{\pm}^-=\bigoplus_{n=1}^{\tilde{n}} \bigoplus_{\alpha>\frac{n^2}{(\lambda_0 \pm \varepsilon)^2}} \cV_A^n (\alpha).$

Taking into consideration the choice of $\lambda_0 \pm \varepsilon$ and the definition of the set $\Lambda(u_0)$ we obtain
\begin{equation*}
\bV_+^- = \bV^-_- \oplus \bigoplus_{n=1}^{\tilde{n}}\bigoplus_{\alpha=\frac{n^2}{\lambda^2_{0}}}\cV_A^n(\alpha)=\colon\bV^-_-\oplus\cV.
\end{equation*}
Therefore
\begin{equation}\label{eq:bifind4}
\begin{split}
&\bif_{T^2} (u_0,\lambda_0)=\\
&=\degttwo((\nabla {\Phi_{\pm}})_{|\bH_0}, B_{\delta}( \bH_0)) \star \mathrm{deg}^{\nabla}_{T^2}(-Id, B_{\delta}(\bV^-_-)) \star(\mathrm{deg}^{\nabla}_{T^2}(-Id, B_{\delta}(\cV))-\bI).\end{split}
\end{equation}
Since $\mathrm{deg}^{\nabla}_{T^2}(-Id, B_{\delta}(\bV^-_-))$ is invertible (see \cite[Theorem 3.11]{GeRy} or \cite[Theorem 2.1]{GolRyb1}), to prove the nontriviality of $\bif_{T^2} (u_0,\lambda_0)$ it is enough to show that
\begin{equation}\label{eq:reducedbif}
\degttwo((\nabla {\Phi_{\pm}})_{|\bH_0}, B_{\delta}( \bH_0)) \star(\mathrm{deg}^{\nabla}_{T^2}(-Id, B_{\delta}(\cV))-\bI)\neq\Theta.
\end{equation}
To that end, note first that since the dimension of $\cV$ is even, by the formula \eqref{eq:degminusid} we obtain
\begin{equation}\label{eq:degidminusid}
\mathrm{deg}^{\nabla}_{T^2}(-Id, B_{\delta}(\cV))-\bI\in U_1(T^2)\oplus U_0(T^2).
\end{equation}
Consider the natural projections
$\pi_i\colon U(T^2)=U_2(T^2)\oplus U_1(T^2)\oplus U_0(T^2) \to U_i(T^2)$ for $i=0,1$.
Using Remark \ref{rem:decomp}(ii) and the formulas \eqref{eq:deg1}, \eqref{eq:degidminusid}, we obtain
\begin{equation*}
\begin{split}
&\degttwo((\nabla {\Phi_{\pm}})_{|\bH_0}, B_{\delta}( \bH_0)) \star(\mathrm{deg}^{\nabla}_{T^2}(-Id, B_{\delta}(\cV))-\bI)=\\&=
\left(\mathfrak{n}_0 \cdot \chi_{T^2}(T^2/T^{2+}) +\sum_{i=1}^{\infty} \mathfrak{n}_{i} \cdot \chi_{T^2}(T^2/H_{(i,0)}^+)\right) \star\\
&\star\left(\pi_1\left(\mathrm{deg}^{\nabla}_{T^2}(-Id, B_{\delta}(\cV))-\bI\right)+\pi_0\left(\mathrm{deg}^{\nabla}_{T^2}(-Id, B_{\delta}(\cV))-\bI\right)\right)=\\&=
\mathfrak{n}_0 \cdot\pi_1\left(\mathrm{deg}^{\nabla}_{T^2}(-Id, B_{\delta}(\cV))-\bI\right)+
\mathfrak{n}_0 \cdot\pi_0\left(\mathrm{deg}^{\nabla}_{T^2}(-Id, B_{\delta}(\cV))-\bI\right)+\\
&+\left(\sum_{i=1}^{\infty} \mathfrak{n}_{i} \cdot \chi_{T^2}(T^2/H_{(i,0)}^+)\right)\star\pi_1\left(\mathrm{deg}^{\nabla}_{T^2}(-Id, B_{\delta}(\cV))-\bI\right).
\end{split}
\end{equation*}
Moreover, using the nontriviality of the representation $\cV$, we get
\begin{equation}\label{eq:nontrivU1}
\pi_1\left(\mathrm{deg}^{\nabla}_{T^2}(-Id, B_{\delta}(\cV))-\bI\right)\neq\Theta.
\end{equation}
Note that
$$
\mathfrak{n}_0 \cdot\pi_0\left(\mathrm{deg}^{\nabla}_{T^2}(-Id, B_{\delta}(\cV))-\bI\right)+\left(\sum_{i=1}^{\infty} \mathfrak{n}_{i} \cdot \chi_{T^2}(T^2/H_{(i,0)}^+)\right)\star\pi_1\left(\mathrm{deg}^{\nabla}_{T^2}(-Id, B_{\delta}(\cV))-\bI\right)
$$
is an element of $U_0(T^2)$ by Remark \ref{rem:decomp}(i). Therefore, if  $\mathfrak{n}_0 \neq 0$, then
$$
\pi_1(\degttwo((\nabla {\Phi_{\pm}})_{|\bH_0}, B_{\delta}( \bH_0)) \star(\mathrm{deg}^{\nabla}_{T^2}(-Id, B_{\delta}(\cV))-\bI))=\mathfrak{n}_0 \cdot\pi_1\left(\mathrm{deg}^{\nabla}_{T^2}(-Id, B_{\delta}(\cV))-\bI\right)\neq \Theta
$$
by the formula \eqref{eq:nontrivU1}. This implies the formula \eqref{eq:reducedbif}.

If, on the other hand, $\mathfrak{n}_0=0$, then the nontriviality of the bifurcation index is equivalent to the nontriviality of
$$
\left(\sum_{i=1}^{\infty} \mathfrak{n}_{i} \cdot \chi_{T^2}(T^2/H_{(i,0)}^+)\right)\star\pi_1\left(\mathrm{deg}^{\nabla}_{T^2}(-Id, B_{\delta}(\cV))-\bI\right).
$$
Note that from the formula \eqref{eq:degminusid} it follows that all the nonzero coefficients of the latter factor of this product are negative.
Hence, using the formula \eqref{eq:nontrivU1} and Lemma \ref{lem:intersections3} we obtain that
this product is nontrivial.

Therefore we have proved that $\bif_{T^2} (u_0,\lambda_0) \neq \Theta$, which finishes the proof.

\section{Discussion of the results}
\label{frc}

In this section we give some remarks concerning our main result. We start with discussing the assumption (a3.2). We also show an example of the system, for which the assumptions (a1)-(a3) are satisfied. Next, we use this example to compare our result with the results concerning Hamiltonian systems given in \cite{danryb}. Moreover, for this example, we compute the precise values of the bifurcation indexes. As a consequence we discuss the topological properties of the bifurcating continua of solutions.

\subsection{Equivariant versus Brouwer index}
\label{ceqb}

As we have mentioned in the introduction, our result can be understood as a generalization of Theorem 3.3 of \cite{danryb}. In particular our assumption (a3.2) generalizes the condition $\ib(u_{i_0}, -U') \neq 0.$ In the following we compare these conditions.

\begin{Remark}\label{rem:degree_comp}
The assumption (a3.2), which seems technical, can be easily verified in some cases. First of all, from the properties of the degree for equivariant gradient maps, if the critical point $u_0 \in (U')^{-1}(0)$ is nondegenerate, then for $\delta>0$ sufficiently small the degree $\deg^{\nabla}_{S^1}(-U',B_{\delta}(u_0,\bR^N))$ is nontrivial in $U(\sone)$, i.e. the assumption (a3.2) is satisfied. The same follows for a  degenerate critical point, provided $\deg_{B}(-U',B_{\delta}(u_0,\bR^N),0)$ is nonzero. Indeed, by the definition of the degree for $\sone$-equivariant gradient maps, the coefficient of $\deg^{\nabla}_{\sone}(-U',B_{\delta}(u_0,\bR^N))$ corresponding to the generator $\chi_{S^1}(S^1/S^{1 +})$ equals $\deg_{B}(-U',B_{\delta}(u_0,\bR^N)^{S^1},0)$, see \cite{Geba}. Moreover, by \cite[Lemma 5.1]{Rabier},
\begin{equation}\label{eq:rem21}
\deg_{B}(-U',B_{\delta}(u_0,\bR^N)^{S^1},0)=\deg_{B}(-U',B_{\delta}(u_0,\bR^N),0).
\end{equation}
Therefore if  $\deg_{B}(-U',B_{\delta}(u_0,\bR^N),0)\neq 0$, then the assumption (a3.2) is satisfied.
\end{Remark}

To discuss some issues of our result we discuss an example of the system \eqref{eq:newtonian} with a particular potential. Namely, we consider an orthogonal $S^1$-representation $\bV=\bR[1,1]\oplus\bR[2,0]$ and define a potential $U\colon\bV\to\bR$ by
\begin{equation}\label{eq:exampU}
U(x_1,x_2,x_3,x_4)=-x_3(x_1^2+x_2^2)+\frac13 x_3^3+\frac12\left(x_4-x_3\right)^2.
\end{equation}

In the lemma below we collect some properties of this system.

\begin{Lemma}\label{lem:exam}
For $U$ given by the formula \eqref{eq:exampU},
\begin{enumerate}[(i)]
\item $(U')^{-1}(0)=\{0\}$,
\item  $\sigma(U''(0))=\{0,2\}$,
\item  $\db (-U',B_{1}(0,\bV),0)=0\in\bZ$,
\item $\deg^{\nabla}_{S^1}(-U',B_{1}(0,\bV))=\chi_{S^1}(S^1/\bZ_1^+) \neq \Theta\in U(\sone)$.
\end{enumerate}
\end{Lemma}
\begin{proof}
We are going to prove (iii) and (iv), the proofs of (i) and (ii) are straightforward.

Consider a family of $S^1$-invariant potentials $H\colon\bV\times[0,1]\to\bR$ given by
\begin{equation}\label{eq:homot0}
H((x_1,x_2,x_3,x_4),t)=-x_3\left(x_1^2+x_2^2-\frac{t}{4}\right)+\frac13 x_3^3+\frac12\left(x_4-(1-t)x_3\right)^2.
\end{equation}
Let $H_t(\cdot)=H(\cdot,t)$. By direct calculations we obtain
\begin{equation}\label{eq:Hadmis}
(H_t')^{-1}(0)=\left\{(x_1,x_2,0,0)\in\bV\colon x_1^2+x_2^2=\frac{t}{4}\right\}.
\end{equation}
In particular, the homotopy $H$ is $B_{1}(0,\bV)$-admissible, i.e. $H_t$  does not have critical points on $\partial B_{1}(0,\bV)$ for any $t \in [0,1]$. Hence, by the homotopy invariance property of the Brouwer degree, we obtain
\[
\db(-U',B_{1}(0,\bV),0)=\db(-H_0',B_{1}(0,\bV),0)=\db(-H_1',B_{1}(0,\bV),0),
\]
where $$H_1(x_1,x_2,x_3,x_4)=-x_3\left(x_1^2+x_2^2-\frac14\right)+\frac13 x_3^3+\frac12 x_4^2.$$
Analogously as in the formula \eqref{eq:rem21}, we have
$$\db(-H_1',B_{1}(0,\bV),0)=\db(-{H'_1}_{|\bV^{S^1}},B_{1}(0,\bV^{S^1}),0).
$$
It is easy to see that $\bV^{S^1}=\bR[2,0]$ and that ${H_1}_{|\bV^{S^1}}(0,0,x_3,x_4)=\frac{1}{3} x_3^3+\frac{1}{4}x_3+\frac12 x_4^2$.
Since this potential does not have any critical points in $\cl (B_1(0,\bV^{S^1}))$, we have $\db({H'_1}_{|\bV^{S^1}}, B_1(0,\bV^{S^1}),0)=0$, which completes the proof of (iii).

To prove (iv) we use again the family of potentials given by the formula \eqref{eq:homot0}. Note that $$H_t' \colon (\cl (B_1(0,\bV)), \partial B_1(0,\bV) ) \to (\bV,\bV\setminus \{0\}), t \in [0,1]$$ is $\sone$-equivariant gradient homotopy. Therefore, by the homotopy invariance property of the degree   for $\sone$-equivariant gradient maps we have
\[\deg^{\nabla}_{S^1}(-U',B_1(0,\bV))=\deg^{\nabla}_{S^1}(-H_0',B_1(0,\bV))=\deg^{\nabla}_{S^1}(-H_1',B_1(0,\bV)).
\]
We will compute $\deg^{\nabla}_{S^1}(-H_1',B_{1}(0,\bV))$ using the definition given in \cite{Geba}.
To that end note first that $(H_1')^{-1}(0)$ consists of exactly one $S^1$-orbit, i.e.
$$(H_1')^{-1}(0)=\left\{(x_1,x_2,0,0)\in\bV\colon x_1^2+x_2^2 =\frac{1}{4}\right\}=S^1\left(P_0\right),$$
where $P_0 \in (H_1')^{-1}(0)$ is any critical point of $H_1$. Without loss of generality, we set $P_0=(0,\frac12,0,0)$.
Then
\begin{itemize}
\item $\dim\ker H_1''(P_0)=\dim S^1(P_0)=1$, i.e. the critical orbit is
nondegenerate.
Indeed
\[
H_1''\left(P_0\right)=\left[
\begin{array}{rrrr}
0&0&0&0\\
0&0&-1&0\\
0&-1&0&0\\
0&0&0&1
\end{array}
\right]
\]
and hence $\dim\ker H_1''(P_0)=1$.
\item $S^1(P_0)\subset \bV_{(\bZ_1)}=\{v\in\bV\colon (S^1_v)=(\bZ_1)\}=\left\{(x_1,x_2,x_3,x_4)\in\bV\colon x_1^2+x_2^2\neq 0\right\}.$
\end{itemize}

Since $\bV_{(\bZ_1)}$ is an open subset of $\bV$, the pair $(\nabla H_1,B_{1}(0,\bV))$ is generic in the sense of Definition 3.1 of \cite{Geba}. Hence, by the formula $(3.5)$ of \cite{Geba} we obtain
\[ \deg^{\nabla}_{S^1}(-H_1',B_{1}(0,\bV))=(-1)^{m^-(-H_1''(P_0))}\chi_{S^1}(S^1/\bZ_1^+)=\chi_{S^1}(S^1/\bZ_1^+), \]
where $m^-(-H_1''(P_0))=2$ is the Morse index of $-H_1''(P_0)$. The proof of (iv) is completed.
\end{proof}

\begin{Remark}
\label{rembg}
The above lemma shows that the assumption (a3.2) covers more cases than the ones discussed in Remark \ref{rem:degree_comp}.
Indeed, considering the potential \eqref{eq:exampU}, we observe that
in this situation assumptions (a1)-(a3) are satisfied. Namely, it is clear that $U$ is $S^1$-invariant and $(U')^{-1}(0)=\{0\}$. Moreover, in Lemma \ref{lem:exam}  we prove that the degree for $\sone$-equivariant gradient maps  $\deg^{\nabla}_{S^1}(-U',B_{\delta}(0,\bR^N))$ is nontrivial, hence the assumption (a3.2) is satisfied. On the other hand, the Brouwer degree $\deg_{B}(-U',B_{\delta}(u_0,\bR^N),0)$ is trivial.


\end{Remark}
%

For Hamiltonian systems obtained from Newtonian ones, our result generalizes the one given in the paper \cite{danryb}. We discuss it in the following remark. Recall first that the Newtonian system \eqref{eq:newtonian} can be translated into a Hamiltonian system of the form
\begin{equation}\label{eq:utoh}
\left\{
\begin{array}{rcc}\dot{u}(t) &=& \lambda v(t),\\
\dot{v}(t) &= &- \lambda U'(u(t)),
\end{array}\right.
\end{equation}
i.e. $\dot x(t)=\lambda JH'(x(t))$, where $x=(u,v)$, $H(u,v)=-\frac{1}{2} |v|^2-U(u)$ and $J= \left[\ba{rr} \Theta & -Id \\Id & \Theta\ea\right]$.

\begin{Remark}
Consider the system \eqref{eq:newtonian} with the potential $U$ and $u_0 \in (U')^{-1}(0)$ satisfying the assumptions (a1)-(a3). In particular $u_0$ is an isolated critical point of $U$, and consequently $(u_0,0)$ is an isolated critical point of $H$. Hence one can define the following index of $(u_0,0) \in (H')^{-1}(0)$:
$$
\ba{rcl}
\ib((u_0,0),H') & = & \db(H',B_{\delta}((u_0,0),\bV\oplus\bV),0) \in \bZ,\\
\ea
$$
for $\delta>0$ as in the assumption (a3.2).
By the product formula of the Brouwer degree we obtain
$$
\ib((u_0,0),H') =\deg_B(-Id,B_{\delta}(\bV),0) \cdot \deg_B(-U',B_{\delta}(u_0,\bV),0).
$$
Hence if $ \ib((u_0,0),H')$ is nontrivial, so is $\deg_B(-U',B_{\delta}(u_0,\bV),0).$  Consequently, analogously as it is discussed in Remark \ref{rem:degree_comp}, $\deg^{\nabla}_{S^1}(-U',B_{\delta}(u_0,\bV))\neq 0$. Using this fact, it can be shown that if the bifurcation index defined by the formula (3.5) of \cite{danryb} (considered for the system \eqref{eq:utoh}) is nontrivial, so is the bifurcation index given in \eqref{eq:indbif}. Hence, bifurcations of nonstationary $2\pi$-periodic solutions of the system \eqref{eq:utoh} obtained by methods of \cite{danryb} are also detected by our main result.

On the other hand, the opposite statement does not hold. To see that, let us consider the potential $U$ given by the formula \eqref{eq:exampU}. Then the bifurcation index considered in \cite{danryb} is trivial as it is a product with one of the factors equal to $\ib((u_0,0),H')$. By the above computations this index is trivial, since $\deg_B(-U',B_{\delta}(u_0,\bV),0)$ is trivial, see Lemma \ref{lem:exam}.
Again from this lemma it follows that the assumptions (a1)-(a3) are satisfied, including $\deg^{\nabla}_{S^1}(-U',B_{1}(0,\bV))\neq \Theta$.
Therefore, by Theorem \ref{thm:main} for every $\lambda\in\Lambda(0)$, from the point $(0,\lambda)$ there bifurcates a continuum of nontrivial solutions. Let us emphasize that these continua are not detected by the methods considered in \cite{danryb}.



\end{Remark}

\subsection{Existence of noncompact continua}

\label{ncco}

In our paper we study the phenomenon of global bifurcation of solutions. Let us recall that this means the existence of a continuum of $2\pi$-periodic solutions of the system, $\cC(u_{0},\lambda_0)\subset C_{2 \pi}([0,2\pi],\Omega) \times (0,+\infty)$, which is either not compact or it is compact and it meets again the family of stationary solutions. Therefore, a natural question arises: can we exclude one of the possibilities of this alternative? Below we will discuss this question. To this end we use the following stronger version of Theorem \ref{thm:biF_in_H}, see Theorem 4.9 of \cite{Ryb2005milano}.



\begin{Theorem}\label{thm:altRab_sum}
If $\bif_{T^2} (u_{0},\lambda_0)\neq \Theta\in U(T^2)$, then a global bifurcation of solutions of the equation \eqref{eq:crit} occurs from $(u_{0},\lambda_0)$. If moreover the continuum $\cC(u_{0},\lambda_0)$ is compact then the set $\cC(u_0, \lambda_0) \cap \cT$ is finite and
\begin{equation}\label{eq:sub_bif}
\sum_{(\widehat{u},\widehat{\lambda})\in \cC(u_0, \lambda_0) \cap \cT}\bif_{T^2}(\widehat{u},\widehat{\lambda})= \Theta.
\end{equation}
\end{Theorem}

In particular, this theorem can be used to study in more details the topological structure of the continuum $\cC(u_{0},\lambda_0)$, provided we have more information about the bifurcation indexes. Below we will discuss an example for which the formula \eqref{eq:sub_bif} never holds and therefore the bifurcating continua are not compact.


Consider the system \eqref{eq:newtonian} with $U$ given by the formula \eqref{eq:exampU}.
Since the only positive eigenvalue of $U''(0)$ is $2$, see Lemma \ref{lem:exam}, the necessary condition given in Fact \ref{fact:necessary} implies that the bifurcation phenomenon can occur only at $(0,\lambda)$ for $\lambda\in\Lambda(0)=\{\frac{k}{\sqrt2}\colon k\in\bN\}.$
From Theorem \ref{thm:main} it follows that for any $k\in\bN$ at $(0,\frac{k}{\sqrt2})$ there occurs a global bifurcation, i.e. there is a continuum of
nontrivial solutions emanating from any such a level. To show that the formula \eqref{eq:sub_bif} cannot hold we will compute the exact values of the bifurcation indexes at all the levels $(0,\frac{k}{\sqrt2})$ for $k\in\bN$. 

\begin{Lemma}\label{lem:bifex}
Consider the system \eqref{eq:newtonian} with $U$ given by the formula \eqref{eq:exampU} and fix $k_0\in\bN$. Then
$$\bif_{T^2} \left(0,\frac{k_0}{\sqrt2}\right)=-\chi_{T^2}\left(T^2/\{e\}^+\right)\in U(T^2).$$
\end{Lemma}

\begin{proof}
Throughout the proof we use the notation of the proof of Theorem \ref{thm:main}.
We will compute
\begin{equation*}
\bif_{T^2} \left(0,\frac{k_0}{\sqrt2}\right)=
\mathrm{deg}^{\nabla}_{T^2}\left(\nabla \Phi_+, B_{\delta}\left(\hone\right)\right)-\mathrm{deg}^{\nabla}_{T^2}\left(\nabla \Phi_-, B_{\delta}\left(\hone\right)\right),
\end{equation*}
where $\Phi_{\pm}=\Phi(\cdot,\frac{k_0}{\sqrt2}\pm\varepsilon)$ and $\Phi$ is  given by the formula \eqref{eq:Phi}.
Recall that
\begin{equation*}
\mathrm{deg}^{\nabla}_{T^2}\left(\nabla \Phi_{\pm}, B_{\delta}( \hone)\right) =\degttwo((\nabla {\Phi_{\pm}})_{|\bH_0}, B_{\delta}( \bH_0)) \star\mathrm{deg}^{\nabla}_{T^2}((\cL_{\pm})_{|\cR_2}, B_{\delta}(\cR_2)),
\end{equation*}
where $\cL_{\pm}=\nabla^2\Phi_{\pm}(0)$, see the formula \eqref{eq:splitted}.
By Lemma \ref{lem:exam}(iv) and the formula \eqref{eq:deg1},
\[\degttwo((\nabla {\Phi_{\pm}})_{|\bH_0}, B_{\delta}( \bH_0))=\chi_{T^2}(T^2/H_{(1,0)}^+).\]
On the other hand,
\[
 \mathrm{deg}^{\nabla}_{T^2}((\cL_{\pm})_{|\cR_2}, B_{\delta}(\cR_2))=\mathrm{deg}^{\nabla}_{T^2}(-Id, B_{\delta}(\bV^-_{\pm})),
\]
where $\bV_{\pm}^-$ is the direct sum of eigenspaces corresponding to negative eigenvalues of $\cL_{\pm}$.
Recall that the eigenvalues of $\cL_{\pm}$ are $\frac{n^2-(\lambda_0 \pm \varepsilon)^2 \alpha}{n^2+1}$.
 Therefore, taking in \eqref{eq:v-+} the values $\alpha=2$  and $(\lambda_0 \pm \varepsilon)=\frac{k_0}{\sqrt2}\pm \varepsilon$, we obtain
$$
\bV_{\pm}^-= \bigoplus_{n<k_0\pm \sqrt2 \varepsilon}\cV_A^n (2),$$
where $A=U''(0)$. In particular $\bV_+^- = \bV^-_- \oplus \cV$, where $\cV= \cV_A^{k_0}(2)$.

To compute $\mathrm{deg}^{\nabla}_{T^2}(-Id, B_{\delta}(\bV^-_{\pm}))$ we will describe the spaces $\cV_A^n (2)$ as $T^2$-representations.
It is easily seen that the eigenspaces of $A$ in $\bR^4$ are $S^1$-representations of the forms
\[
V_A(0)\equiv_{S^1}\bR[1,0]\oplus\bR[1,1],\ \ V_A(2)\equiv_{S^1}\bR[1,0].
\]
Hence $\cV^n_A(2)\subset\bH_n$ is a $T^2$-representation equivalent to $\bR[1,(0,n)]$
as in the formula \eqref{eq:opisVA}.
Summing up, 
\[
\bV^-_- \equiv_{T^2}\bigoplus^{k_0-1}_{n=1} \bR[1,(0,n)], \ \ \cV\equiv_{T^2} \bR[1,(0,k_0)].
\]
Consequently, by the product formula of the degree,
\[
\mathrm{deg}^{\nabla}_{T^2}(-Id, B_{\delta}(\bV^-_{-}))=
\prod_{n=1}^{k_0-1}\mathrm{deg}^{\nabla}_{T^2}(-Id, B_{\delta}(\bR[1,(0,n)])).
\]
Moreover, by the formula \eqref{eq:degminusid2},
$$
\mathrm{deg}^{\nabla}_{T^2}(-Id, B_{\delta}(\bR[1,(0,n)]))=\chi_{T^2}(T^2/T^{2+})- \chi_{T^2}(T^2/H_{(0,n)}^+).
$$
Taking into account the above equalities and the formula \eqref{eq:UT2multiplication}, since $\dim H_{(0,n_1)}\cap H_{(0,n_2)}=\dim H_{(0,gcd(n_1,n_2)})= 1$ for any $n_1,n_2\in\bN$,
$$
\mathrm{deg}^{\nabla}_{T^2}(-Id, B_{\delta}(\bV^-_{-}))
=\chi_{T^2}(T^2/T^{2+})-\sum_{n=1}^{k_0-1} \chi_{T^2}(T^2/H_{(0,n)}^+).
$$
Analogously,
\[
\mathrm{deg}^{\nabla}_{T^2}(-Id, B_{\delta}(\cV))=\chi_{T^2}(T^2/T^{2+})-\chi_{T^2}(T^2/H_{(0,k_0)}^+).
\]
Finally, using the formula \eqref{eq:UT2multiplication},
 \begin{equation*}
\begin{split}
&\bif_{T^2} (0,\frac{k_0}{\sqrt2})=\\
&=\degttwo((\nabla {\Phi_{\pm}})_{|\bH_0}, B_{\delta}(u_0, \bH_0)) \star \mathrm{deg}^{\nabla}_{T^2}(-Id, B_{\delta}(\bV^-_-)) \star(\mathrm{deg}^{\nabla}_{T^2}(-Id, B_{\delta}(\cV))-\bI)=\\
&=\chi_{T^2}(T^2/H_{(1,0)}^+) \star \left(\chi_{T^2}(T^2/T^{2+})-\sum_{n=1}^{k_0-1} \chi_{T^2}(T^2/H_{(0,n)}^+)\right)\star \left(-\chi_{T^2}(T^2/H_{(0,k_0)}^+)\right)=\\
&=\chi_{T^2}(T^2/H_{(1,0)}^+)\star\left(-\chi_{T^2}(T^2/H_{(0,k_0)}^+)\right)=-\chi_{T^2}(T^2/\{e\}^+).
\end{split}
\end{equation*}
\end{proof}

\begin{Corollary}
Consider the system \eqref{eq:newtonian} with $U$ given by the formula \eqref{eq:exampU}. Then, for every $k\in\bN$, the continuum $\cC(0, \frac{k}{\sqrt2})$ is not compact in $C_{2 \pi}([0,2\pi],\Omega) \times (0,+\infty)$.
\end{Corollary}

\begin{proof}
Fix $k\in\bN$. By Theorem \ref{thm:main} and Lemma \ref{lem:exam}, $\cC(0, \frac{k}{\sqrt2}) \neq \{(0, \frac{k}{\sqrt2})\}$. Suppose that this continuum is compact. Then, by Theorem \ref{thm:altRab_sum}, there exist $\frac{k_1}{\sqrt2}, \ldots, \frac{k_r}{\sqrt2}\in \Lambda(0)$ such that $\cC(0, \frac{k}{\sqrt2}) \cap \cT=\bigcup_{j=1}^{r} \{(0,\frac{k_j}{\sqrt2})\}$ and
$
\sum_{j=1}^{r} \bif_{T^2}(0, \frac{k_j}{\sqrt2})= \Theta.
$
On the other hand, by Lemma \ref{lem:bifex}, $\sum_{j=1}^{r} \bif_{T^2}(0, \frac{k_j}{\sqrt2})=-r\cdot \chi_{T^2}(T^2/\{e\}^+)\neq \Theta$. This contradiction finishes the proof.
\end{proof}

Let us emphasize that a similar computation can be done for the system \eqref{eq:newtonian} with any fixed potential. However, giving a general formula for bifurcation indexes, although feasible, would be complicated due to the complex structure of the Euler ring of the torus. On the other hand, even without a general formula, we can give simple conditions implying non-compactness of the bifurcating continua. Let us illustrate it.

\begin{Theorem}\label{thm:unbounded}
Consider the system \eqref{eq:newtonian} with the potential $U$ satisfying the assumptions (a1)-(a3). Assume additionally that $(U')^{-1}(0)=\{u_0\}$. If
\begin{equation*}
\deg^{\nabla}_{S^1}(-U',B_{\delta}(u_0,\bR^N))=
\mathfrak{n}_0 \cdot \chi_{S^1}(S^1/S^{1 +}) +\sum_{i=1}^{\infty} \mathfrak{n}_{i} \cdot \chi_{S^1}(S^1/\bZ_{i}^+)
\end{equation*}
is such that one of the conditions is satisfied:
\begin{enumerate}
\item[(c1)] $\mathfrak{n}_0\neq 0$,
\item[(c2)] $\mathfrak{n}_0=0$ and all nonzero values $\mathfrak{n}_{i}$ are of the same sign,
\end{enumerate}
then all the bifurcating continua are not compact in $C_{2\pi}( \Omega) \times (0,+\infty)$.
\end{Theorem}

\begin{proof}

Fix $\lambda_0\in\Lambda(u_0)$.
 By Theorem \ref{thm:main}, $\cC(u_0, \lambda_0) \neq \{(u_0,\lambda_0)\}$. Observe that from Fact \ref{fact:necessary} it follows that $\cC(u_0, \lambda_0) \cap \cT \subset \{u_0\} \times \Lambda(u_0).$  Suppose, contrary to our claim, that the continuum $\cC(u_0, \lambda_0)$ is compact. Then, by Theorem \ref{thm:altRab_sum}, the set $\cC(u_0, \lambda_0) \cap \cT$ is finite and
\begin{equation}\label{eq:bif_sum1}
\sum_{(u_0,\widehat{\lambda})\in \cC(u_0, \lambda_0) \cap \cT}\bif_{T^2}(u_0,\widehat{\lambda})= \Theta.
\end{equation}

Fix an arbitrary $(u_0, \widetilde{\lambda}) \in \cC(u_0, \lambda_0) \cap \cT.$ Reasoning as in the proof of Theorem \ref{thm:main}, we can obtain the formula for $\bif_{T^2} (u_0,\widetilde{\lambda})$, analogous to the formula \eqref{eq:bifind4}, i.e.
\begin{equation*}
\begin{split}
&\bif_{T^2} (u_0,\widetilde{\lambda})=\\
&=\degttwo((\nabla {\Phi_{\pm}})_{|\bH_0}, B_{\delta}( \bH_0)) \star \mathrm{deg}^{\nabla}_{T^2}(-Id, B_{\delta}(\bV^-_-(\widetilde{\lambda}))) \star(\mathrm{deg}^{\nabla}_{T^2}(-Id, B_{\delta}(\cV(\widetilde{\lambda})))-\bI),\end{split}
\end{equation*}
where $\bV^-_-(\widetilde{\lambda})=\bigoplus_{n=1}^{\infty} \bigoplus_{\alpha >\frac{n^2}{(\widetilde{\lambda}-\varepsilon)^2}} \cV_A^n (\alpha)$, for $\varepsilon$ sufficiently small, and $\cV(\widetilde{\lambda})=\bigoplus_{n=1}^{\infty}\bigoplus_{\alpha=\frac{n^2}{\widetilde{\lambda}^2}}\cV_A^n(\alpha)$.

Note that
\begin{enumerate}
\item[(i)] the formula \eqref{eq:deg1} implies that  $$\degttwo((\nabla {\Phi_{\pm}})_{|\bH_0}, B_{\delta}( \bH_0)) \in \begin{cases} U_2(T^2) \oplus U_1(T^2) & \text{for } \mathfrak{n}_0 \neq 0\\
U_1(T^2) &\text{for } \mathfrak{n}_0 = 0,
\end{cases}$$
\item[(ii)] since the dimension of the representation $\cV(\widetilde{\lambda})$ is even, the formula \eqref{eq:degminusid} implies that $\mathrm{deg}^{\nabla}_{T^2}(-Id, B_{\delta}(\cV(\widetilde{\lambda})))-\bI \in U_1(T^2) \oplus U_0(T^2)$,
\item[(iii)] the same formula implies that $\mathrm{deg}^{\nabla}_{T^2}(-Id, B_{\delta}(\bV^-_-(\widetilde{\lambda}))) \in U_2(T^2) \oplus U_1(T^2) \oplus U_0(T^2).$
\end{enumerate}

Consider the case (c1). Let $\pi_1 \colon U(T^2) \to U_1(T^2)$ be the natural projection. Analysis similar to that in the proof of Theorem \ref{thm:main} shows that from the observations (i)-(iii) above and Remark \ref{rem:decomp} it follows that
$$\pi_1(\bif_{T^2} (u_0,\widetilde{\lambda}))=\mathfrak{n}_0 \cdot \pi_1(\mathrm{deg}^{\nabla}_{T^2}(-Id, B_{\delta}(\cV(\widetilde{\lambda})))-\bI).$$ Using again the formula \eqref{eq:degminusid} and the fact that $\cV(\widetilde{\lambda})$ is a nontrivial $T^2$-representation (since $\widetilde{\lambda} \in \Lambda(u_0)$), we obtain that $\pi_1(\mathrm{deg}^{\nabla}_{T^2}(-Id, B_{\delta}(\cV(\widetilde{\lambda})))-\bI)$ is nonzero and all its nonzero coefficients are negative. This contradicts the formula \eqref{eq:bif_sum1}.

In the case (c2) consider the natural projection $\pi_0 \colon U(T^2) \to U_0(T^2).$ Then, from the observations (i)-(iii) and Remark \ref{rem:decomp} we obtain
$$\pi_0(\bif_{T^2} (u_0,\widetilde{\lambda}))=\pi_1(\degttwo((\nabla {\Phi_{\pm}})_{|\bH_0}, B_{\delta}( \bH_0))) \star \pi_1(\mathrm{deg}^{\nabla}_{T^2}(-Id, B_{\delta}(\cV(\widetilde{\lambda})))-\bI).$$
As in the previous case, we observe that $\pi_1(\mathrm{deg}^{\nabla}_{T^2}(-Id, B_{\delta}(\cV(\widetilde{\lambda})))-\bI)$ is nonzero and all its nonzero coefficients are negative. Taking into consideration the assumption and Corollary \ref{cor:intersections4} we obtain a contradiction with the formula \eqref{eq:bif_sum1}.

\end{proof}

\begin{Remark}
Consider the system \eqref{eq:newtonian} with the potential $U$ satisfying the assumptions (a1)-(a3). From the assumption (a2) it follows that all the critical points of $U$ are fixed points of the action of the group $S^1$ and, in particular, they are isolated in $(U')^{-1}(0)$. On the other hand, we can consider a more general situation when the critical points are not fixed under the action of $\sone$, forming $\sone$-orbits of critical points. Such orbits are homeomorphic to a circle. In this situation we can still consider a bifurcation index, defined in a similar way as it is done in the formula \eqref{eq:indbif}, provided the orbits are isolated sets in $(U')^{-1}(0)$. However, to study and show the nontriviality of such a bifurcation index, one needs to use more advanced methods than the ones applied in this paper.
\end{Remark}

\end{document}